\documentclass[]{siamart250106}
\usepackage{tikz}
\usetikzlibrary{shapes.geometric,calc,chains,arrows.meta,math,intersections}
\usepackage{pgfplots}
\pgfplotsset{compat=newest}  
\usepackage{amsmath,amssymb,amsfonts}
\usepackage{mathtools}
\usepackage{mathrsfs}
\usepackage{graphicx}
\usepackage{subcaption}
\usepackage{comment}

\newtheorem{dfn}{Definition}[section]
\newtheorem{thm}{Theorem}[section]
\newtheorem{lem}{Lemma}[section]
\newtheorem{cor}{Corollary}[section]
\newtheorem{rem}{Remark}[section]
\newtheorem{exam}{Example}[section]

\title{P-order: Unified Convergence Analysis for Nonlinear Iterative Methods}

\author{Xiangmin Jiao\thanks{Department of Applied Mathematics \&
Statistics and Institute for Advanced Computational Science, Stony Brook
University, Stony Brook, NY 11794 (\email{xiangmin.jiao@stonybrook.edu}).}
\and
Hongji Gao\thanks{Department of Applied Mathematics \& Statistics,
Stony Brook University, Stony Brook, NY 11794 (\email{hongji.gao@stonybrook.edu}).}}

\begin{document}
\maketitle

\begin{abstract}
Measuring how quickly iterative methods converge is essential in computational mathematics, but current approaches have significant limitations. Q-order analysis requires strict smoothness conditions, while R-order analysis lacks precision and creates ambiguity, especially when analyzing convergence rates close to linear.

We introduce P-order, a new framework that overcomes these limitations by using a power function $\psi(k)$ combined with asymptotic notation ($\Theta, o, \omega$). Our approach offers two key advantages: it works independently of the chosen norm while providing the precision needed to classify diverse convergence behaviors, including previously hard-to-characterize rates like fractional-power and linearithmic convergence. P-order also systematically accommodates weaker continuity conditions by naturally connecting mathematical assumptions to appropriate Taylor approximation forms.

To enhance practical analysis, we develop two important subclasses, QUP-order and UP-order, which work effectively under different smoothness conditions. We demonstrate P-order's practical value through three applications: (1) refining fixed-point iteration analysis with minimal smoothness requirements (mere differentiability suffices where classical analysis required stronger conditions), (2) identifying previously unreported convergence rates for Newton's method and gradient descent algorithms, and (3) providing a unified analysis of $K$-point methods under $C^{K-1,\nu}$ (i.e., with Hölder continuous $(K-1)$th derivatives), yielding a new characteristic rate $q_{K}(\nu)$.

Our P-order framework provides researchers and practitioners with a sharper, more comprehensive toolbox for convergence analysis, particularly valuable when classical assumptions fail or when analyzing complex convergence behaviors in modern computational applications.
\end{abstract}

\begin{keywords}
Iterative methods, nonlinear equations, optimization, convergence analysis, P-order, H\"older continuity.
\end{keywords}

\begin{AMS}
65H10, 65J15, 65K05, 41A25, 90C30
\end{AMS}

\section{Introduction}
\label{sec:introduction}

Iterative methods are indispensable in scientific computing and machine learning, yet analyzing their convergence rigorously presents challenges. Classical frameworks used for this analysis exhibit profound limitations hindering progress, especially when dealing with the complexities of modern applications. Q-order \cite{traub1964iterative}, while foundational, suffers critically from norm dependence and often requires restrictive smoothness (e.g., $C^q$ continuity \cite{decker1983convergence})---assumptions frequently violated when analyzing methods under weaker conditions like H\"older continuity, common in optimization and machine learning (see, e.g., \cite{berger2020quality,patel2024gradient,ren2010convergence}). R-order \cite{ortega1970iterative}, although norm-independent, sacrifices crucial precision and suffers ambiguities (see \Cref{rem:r-order-ambiguity}). This leaves a pressing need for a unified, norm-independent framework that offers precision across all rates and systematically handles the weaker continuity conditions prevalent today.

To fill this critical gap, we propose P-order (Power-order), designed for precise, norm-independent, and flexible convergence analysis. P-order uses a power function $\psi(k)$ ($\psi(k) \to \infty$) and asymptotic notation ($\Theta, o, \omega$) to characterize rates via $\limsup_{k\to\infty} \|x_k - x_*\|^{1/\psi(k)} = C_\psi \in [0, 1]$. It is norm independent (\Cref{lem:p-norm-independence}) with essential granularity, resolving R-order's ambiguities and allowing sharp distinctions between diverse rates, such as distinguishing linear ($\psi(k)=k$) from linearithmic and anti-linearithmic ($\psi(k)=k\ln k$ and $k/\ln k$, respectively) rates. A key advantage is P-order's systematic handling of weaker continuity by naturally aligning analysis with appropriate Taylor remainder forms. Furthermore, P-order enables enhanced convergence characterizations via QUP-order (Quasi-Uniform, when the limit exists) and UP-order (Uniform, $\Theta$-bounds), aligning the analysis tightly with the problem's known properties. For many analyses, QUP-order offers the most convenient balance of rigor and applicability, and UP-order also offers additional flexibility by replacing $\Theta$ with $\mathcal{O}$ or $\Omega$.

The versatility of the P-order framework enable several novel contributions that advance the convergence theory. First, P-order allows a refined analysis of fixed-point iterations under weakened assumptions; sharp QUP-linear/superlinear rates are established requiring only mere differentiability (\Cref{thm:linear-superlinear}), conditions under which classical Q-order analysis often falters. Moreover, P-order provides the first derivation of explicit conditions for QUP-fractional-power convergence near singularities (\Cref{thm:fp-fractional-power}), characterizing behavior previously unreported. Second, leveraging P-order's precision, we identify and provide the first rigorous characterization of previously unreported convergence rates for Newton's method when applied beyond standard analytic settings \cite{decker1980newton}, including fractional-power, linearithmic, and anti-linearithmic behaviors. We also establish novel fractional-power rates for gradient descent, demonstrating P-order's applicability to core optimization algorithms. Finally, the framework yields a novel, unified analysis of $K$-point methods under relaxed $C^{K-1,\nu}$ H\"older continuity. This analysis culminates in deriving the characteristic UP-order rate $q_K(\nu)$ explicitly as a function of the smoothness exponent (\Cref{thm:Kpoint_holder_rate}), offering unprecedented generality and precision in quantifying the impact of reduced smoothness for general $K$ and $\nu$.

The remainder of the paper is organized as follows. \Cref{sec:background} reviews essential concepts and classical convergence orders. \Cref{sec:pconvergence} introduces the P-order framework. \Cref{sec:fixed-point} presents a refined analysis of general fixed-point iterations, including a sufficient condition for fractional-power convergence. \Cref{sec:examples-sublinear} presents examples for Newton's method and gradient descent with previously unreported convergence rates. \Cref{sec:multipoint} provides a new analysis of interpolation-based multipoint methods under H\"older continuity. \Cref{sec:conclusions} summarizes findings and discusses future work.

\section{Background}
\label{sec:background}

This section reviews essential concepts on vector norms and their equivalences,
fixed-point iterations, the Contraction Mapping Theorem, and the classical convergence
rate concepts of Q-order, R-order, and their variants. Our discussion applies to both the numerical solution of systems of nonlinear equations and multivariate optimization
problems, highlighting the deep connections between these two areas, which often rely on
similar iterative techniques.

\subsection{Vector Norms, Matrix Norms, and Equivalences}
\label{subsec:vector-norms}

In analyzing nonlinear equations and multivariate optimization, vector norms provide essential tools for quantifying errors and convergence rates. The $\ell^p$ norms, defined for $\boldsymbol{x} \in \mathbb{R}^n$ as
\begin{equation}
\|\boldsymbol{x}\|_p = \Biggl(\sum_{i=1}^n |x_i|^p\Biggr)^{1/p}, \quad
1 \le p < \infty,
\label{eq:lp-norm}
\end{equation}
with extension $\|\boldsymbol{x}\|_\infty = \max{1 \le i \le n} |x_i|$, are widely used for their computational efficiency and geometric interpretations. While all $\ell^p$ norms are equivalent in finite-dimensional spaces, the equivalence constants depend on dimension $n$:
\begin{equation}
\alpha(n,p,q)\,\|\boldsymbol{x}\|_p \le \|\boldsymbol{x}\|_q \le
\beta(n,p,q)\,\|\boldsymbol{x}\|_p \quad \text{for all }
\boldsymbol{x} \in \mathbb{R}^n,
\label{eq:norm-equivalence}
\end{equation}
where $\alpha(n,p,q)$ and $\beta(n,p,q)$ are positive constants. This norm dependence is critical when assessing convergence rates in problems
whose dimensions may vary.

Given a matrix $\boldsymbol{A} \in \mathbb{R}^{m \times n}$, the induced matrix norm is defined as
\begin{equation}
\|\boldsymbol{A}\|_p = \sup_{\|\boldsymbol{x}\|_p = 1} \|\boldsymbol{A}\boldsymbol{x}\|_p,
\label{eq:matrix-norm}
\end{equation}
where $\|\boldsymbol{x}\|_p$ is the $\ell^p$ norm of $\boldsymbol{x} \in \mathbb{R}^n$. Similar to vector norms, induced matrix norms are equivalent across different $\ell^p$ norms, where the constants depend on dimensions. While weighted and subordinate norms are sometimes used in convergence analysis, their norm-equivalence constants often depend on matrix numerical values. For clarity, we focus on $\ell^p$ norms.

\subsection{Fixed-Point Iterations and the Contraction Mapping Theorem}
\label{subsec:fixed-point-intro}

Many iterative methods for solving nonlinear equations and optimization problems can be recast as fixed-point iterations. Focusing on the real-valued case (with straightforward extension to complex numbers), such iterations take the form
\[
\boldsymbol{x}_{k+1} = \boldsymbol{g}(\boldsymbol{x}_k),
\]
where \(\boldsymbol{g}\colon \Omega\subseteq\mathbb{R}^n\to\Omega\) is a (possibly nonlinear) mapping. The goal is to find a fixed point \(\boldsymbol{x}_*\) satisfying
\[
\boldsymbol{x}_* = \boldsymbol{g}(\boldsymbol{x}_*).
\]

For instance, consider a system of nonlinear equations
\[
\boldsymbol{f}(\boldsymbol{x}) = \boldsymbol{0}.
\]
Under suitable conditions, this problem can be recast as a fixed-point system by defining
\[
\boldsymbol{x} = \boldsymbol{g}(\boldsymbol{x}) \quad \Longleftrightarrow \quad
\boldsymbol{f}(\boldsymbol{x}) = \boldsymbol{0}.
\]
A common form of these fixed-point iterations is given by
\begin{equation}
\boldsymbol{g}(\boldsymbol{x}) = \boldsymbol{x} - \boldsymbol{A}(\boldsymbol{x})\boldsymbol{f}(\boldsymbol{x}),
\label{eq:fixed-point-equation}
\end{equation}
where $\boldsymbol{A}(\boldsymbol{x})$ typically depends on the inverse of its Jacobian $\boldsymbol{J_f}(\boldsymbol{x})$ (or of its surrogate). Newton's method (also known as the Newton--Raphson method) arises by setting
$\boldsymbol{A}(\boldsymbol{x}) = \boldsymbol{J_f}(\boldsymbol{x})^{-1}$. Similarly, many optimization problems can be recast as fixed-point problems.
For a continuously differentiable function
$f\colon \mathbb{R}^n \to \mathbb{R}$, minimizing $f$ requires
$\boldsymbol{\nabla} f(\boldsymbol{x}_*) = \boldsymbol{0}$.
Expressing this condition as
\begin{equation}
\boldsymbol{x}_* = \boldsymbol{g}(\boldsymbol{x}_*) = \boldsymbol{x}_* - \boldsymbol{A}(\boldsymbol{x}_*)\,\boldsymbol{\nabla} f(\boldsymbol{x}_*),
\label{eq:fixed-point-gradient}
\end{equation}
yields a fixed-point system.
If $\boldsymbol{A}(\boldsymbol{x}) = \alpha\,\boldsymbol{I}$ for $\alpha > 0$, this is the gradient (steepest) descent method.
If $\boldsymbol{A}(\boldsymbol{x}) = \boldsymbol{H}(\boldsymbol{x})^{-1}$, where $\boldsymbol{H}$ is the Hessian of $\boldsymbol{f}$, we recover Newton's method;
if $\boldsymbol{A}(\boldsymbol{x}) = \boldsymbol{B}(\boldsymbol{x})^{-1}$, where $\boldsymbol{B}$ is a surrogate Hessian, we obtain quasi-Newton methods such as BFGS \cite{nocedal2006numerical}.

A key idea in fixed-point convergence analysis is Lipschitz continuity. A function
$\boldsymbol{f}\colon \mathbb{R}^n \to \mathbb{R}^m$ is \emph{Lipschitz continuous} on
$\Omega \subseteq \mathbb{R}^n$ if there is $L \ge 0$ such that
\begin{equation}
\|\boldsymbol{f}(\boldsymbol{x}) - \boldsymbol{f}(\boldsymbol{y})\|
\le L\,\|\boldsymbol{x} - \boldsymbol{y}\|,
\quad \forall\,\boldsymbol{x},\boldsymbol{y} \in \Omega.
\label{eq:Lipschitz}
\end{equation}
The constant $L$, the \emph{Lipschitz constant}, plays a critical role in the convergence analysis of fixed-point iterations. A \emph{contraction mapping} $\boldsymbol{g}$ has Lipschitz constant (under a given norm) $L_{\boldsymbol{g},p} < 1$, so
\begin{equation}
\|\boldsymbol{g}(\boldsymbol{x}) - \boldsymbol{g}(\boldsymbol{y})\|_p
\le L_{\boldsymbol{g},p}\,\|\boldsymbol{x} - \boldsymbol{y}\|_p \quad
\forall\,\boldsymbol{x},\boldsymbol{y}\in \Omega.
\label{eq:contraction}
\end{equation}
We emphasize $p$ to indicate the norm dependence of $L_{\boldsymbol{g},p}$.

The Contraction Mapping Theorem of Banach~\cite{banach1922operations} states
that if $\Omega$ is a closed, complete subset of $\mathbb{R}^n$ and
$\boldsymbol{g}\colon \Omega \to \Omega$ is a contraction mapping, then:
\begin{enumerate}
\item There is a unique fixed point $\boldsymbol{x}_* \in \Omega$.
\item For any $\boldsymbol{x}_0 \in \Omega$, the sequence $\{\boldsymbol{x}_k\}$ defined by $\boldsymbol{x}_{k+1} = \boldsymbol{g}(\boldsymbol{x}_k)$ converges to $\boldsymbol{x}_*$.
\end{enumerate}
Another limitation of Banach's theorem is that it only applies to Lipschitz continuous functions. Many modern applications often involve functions with weaker continuity properties. A function $\boldsymbol{f}\colon \Omega \to \mathbb{R}^m$ is \emph{H\"older continuous} on $\Omega \subseteq \mathbb{R}^n$ with exponent $\nu \in (0, 1]$ and constant $L_\nu > 0$ if
\begin{equation}\label{eq:holder-continuity}
\|\boldsymbol{f}(\boldsymbol{x}) - \boldsymbol{f}(\boldsymbol{y})\| \le L_\nu\,\|\boldsymbol{x} - \boldsymbol{y}\|^\nu
\end{equation}
for all $\boldsymbol{x},\boldsymbol{y}\in\Omega$. We denote by $C^{K,\nu}(\Omega)$ the class of $k$-times continuously differentiable functions whose $k$th derivative, $\boldsymbol{f}^{(k)}$, satisfies this condition. There has been increased interest in analyzing iterative methods for functions with H\"older-continuous derivatives in the context of both solving nonlinear equations~\cite{hernandez2001secant,ren2010convergence} and numerical optimization~\cite{berger2020quality,yashtini2016global}.

More generally, the smoothness of a function can be characterized using moduli of continuity and smoothness~\cite{devore1993constructive}. The \emph{(first) modulus of continuity} of a function $f$ on an interval $I$ is defined as
\begin{equation}
\omega_1(f, t) = \sup_{x,y\in I, |x-y|\le t} |f(x) - f(y)|, \quad t \ge 0.
\label{eq:modulus-continuity}
\end{equation}
A function $f$ is H\"older continuous with exponent $\nu$ if and only if $\omega_1(f, t) = \mathcal{O}(t^\nu)$ as $t\to 0^+$. For $k\ge 1$, the \emph{$k$th modulus of smoothness} is defined using the $k$th forward difference $\Delta_h^k f(x) = \sum_{j=0}^k (-1)^{k-j} \binom{k}{j} f(x+jh)$ as
\begin{equation}
\omega_k(f, t) = \sup_{x, x+kh \in I, |h|\le t} |\Delta_h^k f(x)|, \quad t \ge 0.
\label{eq:modulus-smoothness}
\end{equation}
If $f \in C^{k-1}(I)$, the $k$th modulus of smoothness is related to the modulus of continuity of its $(k-1)$th derivative by the inequality $\omega_k(f, t) \le C t^{k-1} \omega_1(f^{(k-1)}, t)$ for some constant $C$. This relationship is crucial for connecting derivative properties (like H\"older continuity) to the behavior of divided differences, which underpin the analysis of multipoint methods (see \Cref{sec:multipoint}).

\subsection{Q-Order and R-Order}
\label{subsec:q-and-r}

Consider a sequence $\{\boldsymbol{x}_k\}$ converging to $\boldsymbol{x}_*$ in $\mathbb{R}^n$ (or $\mathbb{C}^n$). Two classic frameworks to quantify its convergence rate are Q-order and R-order.

\begin{dfn}[Q-Order]
\label{def:q-order-traub}
A sequence $\{\boldsymbol{x}_k\} \subset \mathbb{R}^n$ converging to $\boldsymbol{x}_*$ has Q-order $q\ge 1$ if there exists $Q_q>0$ (and $0<Q_1<1$ if $q=1$) such that
\begin{equation}
\lim_{k \to \infty} \frac{\|\boldsymbol{x}_{k+1} - \boldsymbol{x}_*\|}{
\|\boldsymbol{x}_k - \boldsymbol{x}_*\|^q} = Q_q.
\label{eq:q-order-def}
\end{equation}
If $q=1$ and $0 < Q_1 < 1$, the convergence is Q-linear; if $q>1$, it is Q-superlinear; if $q=1$ and $Q_1=1$, it is Q-sublinear.
\end{dfn}

This classical definition, originated by Traub~\cite{traub1964iterative} and formalized by Ortega and Rheinboldt~\cite{ortega1970iterative}, is the most widely used in numerical analysis. It is norm-dependent in the sense that the limit may not exist in some (or all) the norms, and it requires asymptotic uniformity of the error.

\begin{rem}[Variants of Q-Order]
\label{rem:eq-order}
Traub developed Q-order by extending the integer-order linear and superlinear convergence of Schr\"oder~\cite{traub1964iterative}. Before Traub's work, Kantorovich~\cite{kantorovich1948functional} employed a specific form to establish Newton's quadratic convergence. 
To relax the asymptotic uniformity requirement in Q-order, Potra and Pták~\cite{potra1984nondiscrete,potra1989q} introduced \emph{``exact'' Q-order} (abbreviated as \emph{EQ-order} hereafter), which holds if there exist $\alpha_q,\beta_q>0$ such that
\begin{equation}
\alpha_q\,\|\boldsymbol{x}_k - \boldsymbol{x}_*\|^q
\;\le\;
\|\boldsymbol{x}_{k+1} - \boldsymbol{x}_*\|
\;\le\;
\beta_q\,\|\boldsymbol{x}_k - \boldsymbol{x}_*\|^q,
\quad k=0,1,\ldots.
\label{eq:exact-q-order}
\end{equation}
This definition omits explicit upper bounds on $\alpha_q$ and $\beta_q$ in the EQ-linear case and keeps them norm-dependent, limiting its usefulness in practice. Nevertheless, we will show that EQ-order is a special case of our proposed P-order framework.
\end{rem}

To address the limitations of Q-order, Ortega and Rheinboldt~\cite{ortega1970iterative} introduced \emph{R-order}. They provided multiple definitions addressing various convergence aspects (e.g., local vs. global, iterative processes vs. single sequences). Subsequent authors (e.g., \cite{dennis1983numerical,nocedal2006numerical,sun2006optimization}) often merged or adapted these notions---sometimes dropping technical details---leading to different ``R-order'' interpretations. To clarify, we adopt Ortega and Rheinboldt's original formulation, but combine their three definitions (~\cite[Section 9.2]{ortega1970iterative}) into a single statement covering all nuances:

\begin{dfn}[R-Order (Ortega and Rheinboldt~\cite{ortega1970iterative})]
\label{def:r-order-OR}
Let $\{\boldsymbol{x}_k\} \subset \mathbb{R}^n$ converge to $\boldsymbol{x}_*$. The \emph{rate convergence factor} or \emph{R-factor} is
\begin{equation}
R_r\{\boldsymbol{x}_k\} =
\begin{cases}
\displaystyle \limsup_{k\to\infty} \|\boldsymbol{x}_k - \boldsymbol{x}_*\|^{1/k}, & r = 1, \\[1mm]
\displaystyle \limsup_{k\to\infty} \|\boldsymbol{x}_k - \boldsymbol{x}_*\|^{1/r^k}, & r > 1.
\end{cases}
\end{equation}
For an iterative process $\mathscr{I}$ with limit $\boldsymbol{x}_*$,
\begin{equation}
R_r\{\mathscr{I},\boldsymbol{x}_*\}
= \sup_{\{\boldsymbol{x}_k\} \in \mathscr{I}} R_r\{\boldsymbol{x}_k\},
\end{equation}
where the supremum is over all convergent sequences in $\mathscr{I}$. The process has \emph{R-order} $O_R$ if
\begin{equation}\label{eq:r-order-OR}
O_R(\mathscr{I}, \boldsymbol{x}_*) =
\begin{cases}
\infty, & \text{if } R_r = 0, \,r\in[1,\infty), \\
\inf_{r\in[1,\infty)}{\{r \mid R_r(\mathscr{I}, \boldsymbol{x}_*) = 1\}}, & \text{otherwise}.
\end{cases}
\end{equation}
If $R_1\in (0,1)$, convergence is \emph{R-linear}; if $R_1=1$ or $0$, it is \emph{R-sublinear} or \emph{R-superlinear}, respectively. Similarly, $R_2\in(0,1)$, $R_2=1$, and $R_2=0$ imply \emph{R-quadratic}, \emph{R-subquadratic}, or \emph{R-superquadratic}, respectively.
\end{dfn}

\Cref{def:r-order-OR} is subtle. It supports non-monotonic sequences and achieves norm independence using the lim-sup of the error norm's root. The distinction between $R_r\{\boldsymbol{x}_k\}$ and $R_r\{\mathscr{I},\boldsymbol{x}_*\}$ concerns single sequences vs. sets of sequences, which we omit (as others did) since a single sequence can serve as the entire set.

\begin{rem}[Differences Between R-order-1 and R-linear]
\label{rem:r-order-ambiguity}
In~\Cref{def:r-order-OR}, R-order-1 is a superset of R-linear. An R-order-1 sequence may still converge R-superlinearly or R-sublinearly, whereas R-linear is mutually exclusive with these two by definition. For example, consider $\|\boldsymbol{x}_k - \boldsymbol{x}_*\| = 0.5^{k\ln k}$, which has $R_r=1$ for all $r>1$ but $R_1=0$. By \Cref{def:r-order-OR}, this sequence has R-order-1 yet is R-superlinear. Similarly, $\|\boldsymbol{x}_k - \boldsymbol{x}_*\| = 0.5^{k/\ln k}$ exhibits R-order-1 and R-sublinear convergence. Analogous behavior occurs for $r>1$, e.g., R-order-2 vs.\ R-quadratic.
\end{rem}

The nuance highlighted in \Cref{rem:r-order-ambiguity} was not stated explicitly by Ortega and Rheinboldt~\cite{ortega1970iterative}, leading to different re-interpretations by Dennis and Schnabel~\cite{dennis1983numerical} (``R-order'') and Nocedal and Wright~\cite{nocedal2006numerical} (``R-linear''). 
The main confusion is that the infimum in \eqref{eq:r-order-OR} does not yield a precise rate, which our P-order framework resolves. Under the original R-order definition from \Cref{def:r-order-OR}, Ortega and Rheinboldt~\cite[Section 3.2]{ortega1972numerical} showed R-order is always at least the Q-order and $R_1$ is at most $Q_1$. Our P-order further refines these relationships for linear-convergent cases, avoiding R-order's ambiguities and enabling more refined characterizations of convergence rates.

\subsection{Asymptotic Notation in Algorithm Analysis}
\label{subsec:asymptotic-notation}

To develop a unified framework for convergence analysis, we draw inspiration from
asymptotic notations used in algorithm analysis. These notations provide a concise
way to describe the growth rates of functions, abstracting away constant factors
and lower-order terms. For functions $f(n)$ and $g(n)$ defined on the positive integers
(or more generally, any unbounded set), we primarily use:

\begin{itemize}
    \item \textbf{Big-Theta:} $f(n) = \Theta(g(n))$ if $f(n)=\mathcal{O}(g(n))$ and $f(n)=\Omega(g(n))$,
    where $f(n) = \mathcal{O}(g(n))$ if there exist constants $c>0$ and $n_0$ such that $f(n) \le c\,g(n)$ for all $n \ge n_0$, and $f(n) = \Omega(g(n))$ if $g(n)=\mathcal{O}(f(n))$.
    \item \textbf{Little-o:} $f(n) = o(g(n))$ if $\lim_{n\to\infty} f(n)/g(n) = 0$.
    \item \textbf{Little-omega:} $f(n) = \omega(g(n))$ if $\lim_{n\to\infty} f(n)/g(n) = \infty$.
\end{itemize}

For instance, $3n^2 + 5n + 2 = \Theta(n^2)$, and $n \ln n = o(n^2)$.
In algorithm analysis, these notions are used to describe the running time or space
complexity of an algorithm as the input size grows; see, e.g., \cite{cormen2009introduction}
and \cite{knuth1997art}. We will use similar ideas in this work to characterize convergence
rates as the iteration count grows.

\section{P-Order: A Unified Framework}
\label{sec:pconvergence}

This section introduces P-order, a unified framework for quantifying the convergence rates of iterative methods, establishing its properties, defining key subclasses, and establishing their relationships with Q-order and R-order.

\subsection{General Framework}

We first define P-order, which is inspired by and generalizes the R-order framework.

\begin{dfn}[P-Order]
\label{def:p-order-general}
Let $\{\boldsymbol{x}_k\}_{k\ge 0} \subset \mathbb{R}^n$ converge to $\boldsymbol{x}_*$. Let $\psi\colon \mathbb{N}\to (0,\infty)$ satisfies $\lim_{k\to\infty}\psi(k)=\infty$. The sequence converges with \emph{P-order} characterized by the power function $\psi(k)$ if
\begin{equation}\label{eq:p-conv-def}
\limsup_{k\to\infty}\|\boldsymbol{x}_k-\boldsymbol{x}_*\|^{1/\psi(k)} = C_\psi,
\end{equation}
where the constant $C_\psi \in (0, 1)$ is the \emph{P-base}. If $C_\psi=0$, the rate is \emph{P-super-$\psi$}. If $\limsup_{k\to\infty}\|\boldsymbol{x}_k-\boldsymbol{x}_*\|^{1/\psi(k)} = 1$, the rate is \emph{P-sub-$\psi$}. If $C_\psi=0$ for all $\psi(k) \to \infty$, the convergence is \emph{P-order-$\infty$} (finite termination).
\end{dfn}

In \Cref{def:p-order-general}, convergence is asymptotically faster for larger power functions $\psi(k)$; for a given $\psi(k)$, convergence is faster for smaller P-bases $C_\psi \in (0,1)$. The P-order-$\infty$ case represents convergence faster than any finite P-order rate and corresponds to finite termination (i.e., $\|\boldsymbol{x}_k-\boldsymbol{x}_*\|=0$ for some finite $k$). This concept enhances the completeness of subsequent results (e.g., \Cref{thm:linear-superlinear}). An important property of P-order is its norm-independence, established next.

\begin{lem}[Norm Independence]
\label{lem:p-norm-independence}
If a sequence $\{\boldsymbol{x}_k\}$ converges with P-order $\psi(k)$ (which tends to $\infty$) and P-base $C_\psi$ in some norm $\|\cdot\|_p$, it converges with the same P-order $\psi(k)$ and P-base $C_\psi$ in any equivalent norm $\|\cdot\|_q$.
\end{lem}
\begin{proof}
Let $\xi_{k,p} = \|\boldsymbol{x}_k-\boldsymbol{x}_*\|_p$. Due to the norm equivalence \eqref{eq:norm-equivalence}, we have $$\alpha^{1/\psi(k)} (\xi_{k,p})^{1/\psi(k)} \le (\xi_{k,q})^{1/\psi(k)} \le \beta^{1/\psi(k)} (\xi_{k,p})^{1/\psi(k)}.$$ Since $\alpha^{1/\psi(k)} \to 1$ and $\beta^{1/\psi(k)} \to 1$ as $k\to\infty$ (because $\psi(k)\to\infty$), taking the $\limsup$ yields $\limsup_{k\to\infty}(\xi_{k,q})^{1/\psi(k)} = \limsup_{k\to\infty}(\xi_{k,p})^{1/\psi(k)} = C_\psi$. Hence, both $\psi(k)$ and $C_\psi$ are independent of the norm used.
\end{proof}

For the special case of $\psi(k)=\Theta(k)$, P-order is classified as \emph{P-sublinear} or \emph{P-superlinear} if $C_\psi=1$ and $C_\psi=0$, respectively. We can also classify P-sublinear and P-superlinear as follows:
\begin{itemize}
    \item \textbf{P-sublinear}: $\psi(k) = o(k)$.
    \item \textbf{P-linear}: $\psi(k) = k$. Here, $C_\psi \in (0,1)$ is the \emph{asymptotic error constant}.
    \item \textbf{P-superlinear}: $\psi(k) = \omega(k)$.
\end{itemize}

While analogous to those of R-order, the definitions of P-sublinear based on either $C_\psi=1$ (with $\psi(k)=k$) or $\psi(k)=o(k)$, and similarly for P-superlinear based on $C_\psi=0$ (with $\psi(k)=k$) or $\psi(k)=\omega(k)$, are fully consistent, avoiding ambiguities noted for R-order (cf. \Cref{rem:r-order-ambiguity}), as established by the following lemma.

\begin{lem}[P-Sub-$\psi$ and P-Super-$\psi$ Convergence]
\label{lem:p-sub-super-psi}
Let $\{\boldsymbol{x}_k\}$ converge with P-order characterized by $\psi(k)$ and P-base $C_\psi \in (0,1)$. For another power function $\varphi(k)$ satisfying $\varphi(k)\to\infty$: $\psi(k) = o(\varphi(k))$ iff $\limsup_{k\to\infty}\|\boldsymbol{x}_k-\boldsymbol{x}_*\|^{1/\varphi(k)} = 1$, and $\psi(k) = \omega(\varphi(k))$ iff $\limsup_{k\to\infty}\|\boldsymbol{x}_k-\boldsymbol{x}_*\|^{1/\varphi(k)} = 0$.
\end{lem}
\begin{proof}
Let $L = \limsup_{k\to\infty}\|\boldsymbol{x}_k-\boldsymbol{x}_*\|^{1/\varphi(k)}$, which equals $\limsup_{k\to\infty} C_\psi^{\psi(k)/\varphi(k)}$.

(i) For $\psi(k) = o(\varphi(k))$: Since $\lim_{k\to\infty} \psi(k)/\varphi(k) = 0$ and $0 < C_\psi < 1$, we have $L = \limsup_{k\to\infty} C_\psi^{\psi(k)/\varphi(k)} = C_\psi^0 = 1$. Conversely, if $L = 1$, then $\psi(k)/\varphi(k)$ must approach 0 as $k\to\infty$, since any non-zero limit or unbounded growth would yield $L < 1$. Hence $\psi(k) = o(\varphi(k))$.

(ii) For $\psi(k) = \omega(\varphi(k))$: Since $\lim_{k\to\infty} \psi(k)/\varphi(k) = \infty$ and $0 < C_\psi < 1$, we have $L = \limsup_{k\to\infty} C_\psi^{\psi(k)/\varphi(k)} = 0$. Conversely, if $L = 0$, then $\psi(k)/\varphi(k)$ must tend to $\infty$ as $k\to\infty$ since $C_\psi^r \to 0$ only when $r \to \infty$, which implies $\psi(k) = \omega(\varphi(k))$.
\end{proof}

Setting $\varphi(k)=k$ in \Cref{lem:p-sub-super-psi} confirms the consistency noted above between the different ways of classifying P-order convergence. This internal consistency, combined with the use of the power function $\psi(k)$ in a manner analogous to little-$o$, big-$\Theta$, and little-$\omega$ notation from algorithm analysis, makes P-order a much more refined framework than R-order. In particular, P-order can distinguish different sublinear and superlinear rates, including potentially exotic rates such as linearithmic or anti-linearithmic convergence (as we will discuss in \Cref{subsec:common-p-rates}). Furthermore, the constant factor within $\psi(k)$ (e.g., when $\psi(k) = c\,k^r$) can further refine the classification for detailed comparative studies if needed.

\subsection{Quasi-Uniform P-Order}

The general P-order definition uses $\limsup$, which accommodates non-uniform convergence but can complicate analysis. When the limit exists, we have the simpler Quasi-Uniform P-order (QUP-order).

\begin{dfn}[Quasi-Uniform P-Order (QUP-Order)]
\label{def:p-quasi-uniform-order}
A sequence $\{\boldsymbol{x}_k\} \subset \mathbb{R}^n$ converging to $\boldsymbol{x}_*$ has \emph{QUP-order} with power function $\psi(k)$ (where $\psi(k)\to\infty$) and P-base $C_\psi\in (0,1)$ if the limit exists, i.e.,
\begin{equation}\label{eq:p-quasi-uniform-conv-def}
\lim_{k\to\infty}\|\boldsymbol{x}_k-\boldsymbol{x}_*\|^{1/\psi(k)} = C_\psi.
\end{equation}
An equivalent form of QUP-order is that for any $\epsilon > 0$,
\begin{equation}\label{eq:p-quasi-uniform-conv-alt}
\omega\Bigl((C_\psi-\epsilon)^{\psi(k)}\Bigr) =
    \|\boldsymbol{x}_k-\boldsymbol{x}_*\| = o\Bigl((C_\psi+\epsilon)^{\psi(k)}\Bigr).
\end{equation}
\end{dfn}

The limit form \eqref{eq:p-quasi-uniform-conv-def} and the asymptotic-bound form \eqref{eq:p-quasi-uniform-conv-alt} of QUP-order are equivalent, as shown below.

\begin{lem}[Equivalence of QUP-Order Forms]
\label{lem:qup-order}
\Cref{eq:p-quasi-uniform-conv-def} holds iff \eqref{eq:p-quasi-uniform-conv-alt} holds.
\end{lem}
\begin{proof}
Let $\xi_k = \|\boldsymbol{x}_k-\boldsymbol{x}_*\|$.

\textbf{(i) $\eqref{eq:p-quasi-uniform-conv-def} \implies \eqref{eq:p-quasi-uniform-conv-alt}$:} Assume $\lim_{k\to\infty} \xi_k^{1/\psi(k)} = C_\psi$. For any $\epsilon > 0$, consider the limit of the $k$th root of $\xi_k / (C_\psi+\epsilon)^{\psi(k)}$,
\[ \lim_{k\to\infty} \left( \frac{\xi_k}{(C_\psi+\epsilon)^{\psi(k)}} \right)^{1/\psi(k)} = \lim_{k\to\infty} \frac{\xi_k^{1/\psi(k)}}{C_\psi+\epsilon} = \frac{C_\psi}{C_\psi+\epsilon} < 1. \]
By the root test for sequences, this implies $\lim_{k\to\infty} \xi_k / (C_\psi+\epsilon)^{\psi(k)} = 0$, meaning $\xi_k = o((C_\psi+\epsilon)^{\psi(k)})$.
Similarly,
\[ \lim_{k\to\infty} \left( \frac{\xi_k}{(C_\psi-\epsilon)^{\psi(k)}} \right)^{1/\psi(k)} = \frac{C_\psi}{C_\psi-\epsilon} > 1. \]
The root test implies $\lim_{k\to\infty} \xi_k / (C_\psi-\epsilon)^{\psi(k)} = \infty$, meaning $\xi_k = \omega((C_\psi-\epsilon)^{\psi(k)})$. These $o/\omega$ conditions correspond precisely to \eqref{eq:p-quasi-uniform-conv-alt}.

\textbf{(ii) $\eqref{eq:p-quasi-uniform-conv-alt} \implies \eqref{eq:p-quasi-uniform-conv-def}$:}
Assume that for every $\epsilon>0$ we have
\[ \omega\Bigl((C_\psi-\epsilon)^{\psi(k)}\Bigr) \le \xi_k \le o\Bigl((C_\psi+\epsilon)^{\psi(k)}\Bigr). \]
By the definitions of $\mathcal{O}(\cdot)$ and $\omega(\cdot)$, this means that for any $\epsilon>0$, there exists an index $K_\epsilon$ such that for all $k > K_\epsilon$,
$(C_\psi-\epsilon)^{\psi(k)} < \xi_k < (C_\psi+\epsilon)^{\psi(k)}$.
Taking the $1/\psi(k)$ root gives
\[ C_\psi-\epsilon < \xi_k^{1/\psi(k)} < C_\psi+\epsilon \quad \text{for all } k > K_\epsilon. \]
Since this holds for any $\epsilon>0$, by definition of the limit, $\lim_{k\to\infty} \xi_k^{1/\psi(k)} = C_\psi$.
\end{proof}

QUP-order ($\lim$ exists) strengthens P-order ($\limsup$). Uniform P-order (UP-order) is stronger still, requiring $\|\boldsymbol{x}_k-\boldsymbol{x}_*\| = \Theta(C_\psi^{\psi(k)})$:

\begin{dfn}[Uniform P-Order (UP-Order)]
\label{def:up-order}
A sequence $\{\boldsymbol{x}_k\}$ has \emph{UP-order} characterized by $\psi(k)$ (where $\psi(k)\to\infty$) with P-base $C_\psi\in (0,1)$ if
\begin{equation}\label{eq:p-uniform-conv-def}
\|\boldsymbol{x}_k-\boldsymbol{x}_*\| = \Theta\Bigl(C_\psi^{\psi(k)}\Bigr).
\end{equation}
\end{dfn}

The $\Theta(\cdot)$ form \eqref{eq:p-uniform-conv-def} for UP-order and the equivalent small-$o/\omega$ form \eqref{eq:p-quasi-uniform-conv-alt} for QUP-order, both featuring $C_\psi^{\psi(k)}$, motivate the name ``Power-order.'' Clearly, UP-order implies QUP-order. While UP-order has closer ties to Q-order (see \Cref{sec:p-q-relationships}), QUP-order is often more practical for analysis. Thus, QUP-order will be the primary focus henceforth.

\begin{rem}[Relaxation of UP-order]
\label{rem:relaxed-up-order}
By definition UP-order is
\[
\Omega((C_\psi)^{\psi(k)}) = \|\boldsymbol{x}_k-\boldsymbol{x}_*\| = \mathcal{O}((C_\psi)^{\psi(k)}).\] Omitting the left and right hand sides would mean \emph{at least} and \emph{at most} UP-order (or P-order) with power function $\psi(k)$, correspondingly. However, unlike QUP-order, UP-order \textbf{cannot} be written in a small-$o/\omega$ form.
\end{rem}

\subsection{Exponential, Polynomial, and Logarithmic Convergence}
\label{subsec:common-p-rates}

We now consider common power functions $\psi(k)$ (where $\psi(k)\to\infty$), analogous to those in algorithm analysis. Their corresponding P-order refines the granularity of Q-order and R-order categorizations. While the P-order framework is general, the following list illustrates key cases using UP-order notation $\|\boldsymbol{x}_k-\boldsymbol{x}_*\| = \Theta(C^{\psi(k)})$ for brevity; these $\psi$ functions apply equally to QUP-order and P-order. In these examples, the P-base $C$ (omitting the subscript $\psi$ for brevity) satisfies $C \in (0,1)$.

\medskip
\noindent
\textbf{Superlinear Exponential P-Order of order $r > 1$}: Set $\psi(k) = r^k$. For UP-order,
\begin{equation}\label{eq:superlinear-poly}
\|\boldsymbol{x}_k-\boldsymbol{x}_*\| = \Theta\Bigl(C^{\displaystyle r^k}\Bigr).
\end{equation}
Consistent with R-superlinear rates, we term this \emph{P-order-$r$} ($O_P=r$), consistent with Q-order for $r>1$ (e.g., P-quadratic for $r=2$; cf. \Cref{thm:q-superlinar-implies-up}). Note $\psi(k)=r^k$ is invalid if $r\leq 1$ as it does not tend to $\infty$.

\medskip
\noindent
\textbf{Linear and Sublinear Fractional-Power P-Order of order $r\leq 1$}: Set $\psi(k)=k^r$ for $r\in(0,1]$. This includes P-linear ($\psi=k$) and P-fractional-power ($0<r<1$). The unified UP-order form is
\begin{equation}\label{eq:frac-power-unified}
\|\boldsymbol{x}_k-\boldsymbol{x}_*\| = \Theta\Bigl(C^{\displaystyle k^r}\Bigr).
\end{equation}
Sublinear fractional-power rates ($0<r<1$) are termed \emph{fractional-power}, based on algorithm analysis terminology. (Note: superlinear polynomial rates $\psi=k^r, r>1$, are slower than exponential $\psi=r^k$ and less common). For consistency with Q-order, P-order-$r$ uses $\psi=\Theta(r^k)$ if $r>1$, but $\psi=\Theta(k^r)$ if $r\leq 1$. Fractional-power rates are asymptotically slower than polynomial but faster than logarithmic (or polylogarithmic) rates.

\medskip
\noindent
\textbf{Logarithmic P-Order}: Set $\psi(k) = \ln k$. For UP-order,
\begin{equation}\label{eq:log-p}
\|\boldsymbol{x}_k-\boldsymbol{x}_*\| = \Theta\Bigl(C^{\ln k}\Bigr) = \Theta\Bigl(k^{\ln C}\Bigr),
\end{equation}
where $\ln C < 0$ since $C<1$ (as $C^{\ln k} = (e^{\ln k})^{\ln C} = k^{\ln C}$).

\medskip
\noindent
\textbf{Linearithmic and Anti-Linearithmic\footnote{There does not appear to be an established term for the $k/\ln k$ rates. We refer to it as ``anti-linearithmic'' to signify that it is the sublinear counterpart of linearithmic $k\ln k$ rates.} P-Order}: Set $\psi(k) = k\ln k$ (linearithmic) and $\psi(k) = k/\ln k$ (anti-linearithmic). For UP-order, they are
\begin{equation}\label{eq:linearithmic-p}
\|\boldsymbol{x}_k-\boldsymbol{x}_*\| = \Theta\Bigl(C^{\displaystyle k\ln k}\Bigr) \quad \text{and} \quad \|\boldsymbol{x}_k-\boldsymbol{x}_*\| = \Theta\Bigl(C^{\displaystyle k/\ln k}\Bigr)
\end{equation}
respectively. For example, linearithmic rates ($\Theta(1/k!)$ error implies this) are P-superlinear but slower than any P-order-$(1+\epsilon)$. Anti-linearithmic is P-sublinear but faster than any P-order-$(1-\epsilon)$, where $\epsilon>0$ is arbitrary small. The same concept applies to other polynomial rates modulated by logarithmic factors, e.g., $\psi(k) = k^2\ln k$ and $\psi(k) = k^2/\ln k$.
\medskip

The logarithmic convergence rate is consistent with established notions. However, fractional-power, linearithmic, and anti-linearithmic rates represent new classifications enabled by P-order's refined framework, to characterize convergence rates under different moduli of continuity and smoothness. The unified $\psi(k)=k^r$ form for $r \in (0,1]$ motivates the fractional-power name, extending the linear case naturally. Other rates can also be defined, but we omit them for brevity.

\Cref{fig:p-order} compares several P-order rates by plotting both $\psi(k)$ (\Cref{fig:p-conv-psi}) and the corresponding normalized error decay $C^{\psi(k)}$ (\Cref{fig:p-conv-error}). Representative rates shown range from logarithmic and fractional-power ($k^{0.5}$) to linear, linearithmic, polynomial ($k^2$), and exponential ($2^k$, P-quadratic). As expected, faster asymptotic growth of $\psi(k)$ yields faster error reduction (\Cref{fig:p-conv-error}). Notably, logarithmic convergence is asymptotically much slower than linear, while the practical error reduction for exponential and polynomial rates appear similar during initial iterations despite the significant difference in their $\psi(k)$ growth. The figure illustrates the wide spectrum of sublinear and superlinear rates distinctly characterized by P-order.

\begin{figure}[htbp]
\begin{subfigure}[b]{0.48\textwidth}
\centering
\begin{tikzpicture}[scale=0.74]
\begin{loglogaxis}[
    xlabel={Iteration ($k$)},
    ylabel={Power Function ($\psi(k)$)},
    xmin=1, xmax=200,  
    ymin=0.5, ymax=3000,  
    major grid style={line width=.2pt,draw=gray!50},
    minor grid style={line width=.1pt,draw=gray!50},
    legend pos=north west,  
    legend style={nodes={scale=0.8, transform shape, anchor=west}, opacity=0.5, text opacity=1},
    legend entries={
        $2^k$,
        $k^2$,
        $k \ln(k)$,
        $k$,
        $\frac{k}{\ln(k + e - 1)}$,
        $k^{0.5}$,
        $\ln(k)$
    },
]


\addplot[domain=1:200, samples=50, smooth, color=red, dotted, thick] {2^x};
\addplot[domain=1:200, samples=50, smooth, color=brown, dashed, thick] {x^2};
\addplot[domain=1:200, samples=50, smooth, color=blue, dashdotted, thick] {x*ln(x)};
\addplot[domain=1:200, samples=50, smooth, color=black, thick] {x};
\addplot[domain=1:200, samples=50, smooth, color=orange, dashdotted, thick] {x/ln(x+exp(1)-1)};
\addplot[domain=1:200, samples=50, smooth, color=purple, dashed, thick] {x^0.5};
\addplot[domain=1:200, samples=50, smooth, color=olive, dotted, thick] {ln(x)};

\end{loglogaxis}
\end{tikzpicture}
\caption{Growth of $\psi(k).$\label{fig:p-conv-psi}}
\end{subfigure}
\hfill
\begin{subfigure}[b]{0.48\textwidth}
\centering
\begin{tikzpicture}[scale=0.74]
\begin{loglogaxis}[
    xlabel={Iteration ($k$)},
    ylabel={Normalized Error ($\|\boldsymbol{x}_k\|/\|\boldsymbol{x}_1\|$)},
    xmin=1, xmax=200,
    ymin=1e-8, ymax=1.2,
    major grid style={line width=.2pt,draw=gray!50},
    minor grid style={line width=.1pt,draw=gray!50},
]

\pgfmathsetmacro{\C}{0.5}

\pgfmathdeclarefunction{logarithmic}{1}{%
  \pgfmathparse{ln(#1)}%
}
\pgfmathdeclarefunction{fractional_power}{1}{%
  \pgfmathparse{#1^0.5}%
}
\pgfmathdeclarefunction{anti_linearithmic}{1}{%
  \pgfmathparse{#1 / ln(#1 + exp(1) - 1)}%
}
\pgfmathdeclarefunction{linear}{1}{%
  \pgfmathparse{#1}%
}
\pgfmathdeclarefunction{linearithmic}{1}{%
  \pgfmathparse{#1 * ln(#1)}%
}
\pgfmathdeclarefunction{quadratic}{1}{%
  \pgfmathparse{#1^2}%
}
\pgfmathdeclarefunction{exponential}{1}{%
  \pgfmathparse{2^#1}%
}

\addplot[domain=1:200, samples=200, smooth, dotted, color=olive, mark=none, thick]
    {(\C^logarithmic(x)) / (\C^logarithmic(1))};
\addplot[domain=1:200, samples=200, smooth, color=purple, dashed, mark=none, thick]
    {(\C^fractional_power(x)) / (\C^fractional_power(1))};
\addplot[domain=1:150, samples=200, smooth, color=orange, dashdotted, mark=none, thick]
    {(\C^anti_linearithmic(x)) / (\C^anti_linearithmic(1))};
\addplot[domain=1:40, samples=200, smooth, color=black, mark=none, thick]
    {(\C^linear(x)) / (\C^linear(1))};
\addplot[domain=1:20, samples=200, smooth, color=blue, dashdotted, mark=none, thick]
    {(\C^linearithmic(x)) / (\C^linearithmic(1))};
\addplot[domain=1:6, samples=10, smooth, color=brown, dashed, mark=none, thick]
    {(\C^(x^2)) / \C};
\addplot[domain=1:6, samples=10, smooth, color=red, dotted, mark=none, thick]
    {(\C^(2^x)) / (\C^2)};

\end{loglogaxis}
\end{tikzpicture}
\caption{$C^{\psi(k)}$ normalized to $\xi_1=1$ with $C=0.5$.\label{fig:p-conv-error}}
\end{subfigure}
\caption{Comparison of P-order with different $\psi(x)$}
\label{fig:p-order}
\end{figure}
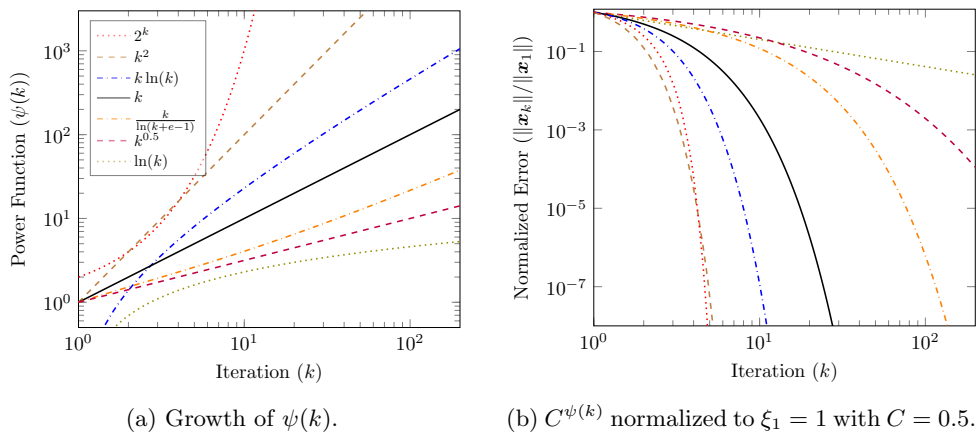

\subsection{Relationship with Q-Order}
\label{sec:p-q-relationships}

Having defined the P-order framework and its subclasses, we now formally establish their relationships with the classical Q-order. These connections, particularly involving QUP-order, offer more precision than comparisons available for R-order \cite[Section 9.3]{ortega1970iterative}.

\subsubsection{Q-Linear vs QUP-Linear Convergence}

We begin by examining the direct link between Q-linear convergence and QUP-linear convergence. The following theorem shows that Q-linear convergence is a sufficient condition for QUP-linear convergence with the same asymptotic rate constant.

\begin{theorem}[Q-Linear Implies QUP-Linear]
\label{thm:q-linear-implies-qup-linear}
If a sequence $\{\boldsymbol{x}_k\}$ converges Q-linearly with factor $Q_1 \in (0,1)$, i.e.,
\[ \lim_{k\to\infty} \frac{\|\boldsymbol{x}_{k+1}-\boldsymbol{x}_*\|}{\|\boldsymbol{x}_k-\boldsymbol{x}_*\|} = Q_1, \]
then it also converges QUP-linearly with P-function $\psi(k)=k$ and P-base $C_\psi=Q_1$.
\end{theorem}
\begin{proof}
Let $\xi_k = \|\boldsymbol{x}_k-\boldsymbol{x}_*\|$. The Q-linear limit definition implies that for any sufficiently small $\epsilon > 0$, there exist constants $A_\epsilon, B_\epsilon > 0$ and an index $K_\epsilon$ such that
\[ A_\epsilon (Q_1 - \epsilon)^k < \xi_k < B_\epsilon (Q_1 + \epsilon)^k \quad \text{for all } k \ge K_\epsilon. \]
Taking the $k$th root yields
\[ A_\epsilon^{1/k} (Q_1 - \epsilon) < \xi_k^{1/k} < B_\epsilon^{1/k} (Q_1 + \epsilon). \]
Since $A_\epsilon^{1/k} \to 1$ and $B_\epsilon^{1/k} \to 1$ as $k \to \infty$, and $\epsilon>0$ can be arbitrarily small, the Squeeze Theorem implies $\lim_{k \to \infty} \xi_k^{1/k} = Q_1$. By \Cref{def:p-quasi-uniform-order}, this signifies QUP-linear convergence with $\psi(k)=k$ and $C_\psi=Q_1$.
\end{proof}

While \Cref{thm:q-linear-implies-qup-linear} shows Q-linear implies QUP-linear, a partial converse holds under the crucial assumption that the Q-limit itself exists, leading to the following relationship between the constants:

\begin{cor}[Partial Converse for QUP-Linear]
\label{cor:qup-linear-implies-q-linear}
If $\{\boldsymbol{x}_k\}$ converges QUP-linearly ($\psi(k)=k$) with P-base $C_\psi$, and the Q-limit $Q_1 = \lim_{k\to\infty} \frac{\|\boldsymbol{x}_{k+1} - \boldsymbol{x}_*\|}{\|\boldsymbol{x}_k - \boldsymbol{x}_*\|}$ exists in some norm, then $Q_1 = C_\psi$.
\end{cor}
\begin{proof}
Existence of the Q-limit $Q_1$ implies, by \Cref{thm:q-linear-implies-qup-linear}, QUP-linear convergence with P-base $Q_1$. Since the P-base is norm-independent (\Cref{lem:p-norm-independence}), this $Q_1$ must equal the sequence's given P-base $C_\psi$.
\end{proof}

\Cref{thm:q-linear-implies-qup-linear} and \Cref{cor:qup-linear-implies-q-linear} highlight some important distinctions between different notions of linear convergence: First, Q-linear convergence (and thus QUP-linear convergence, by \Cref{thm:q-linear-implies-qup-linear}) does \emph{not} guarantee the stronger UP-linear condition ($\xi_k = \Theta(C_\psi^k)$). UP-linear forbids significant oscillations relative to the base rate $C_\psi^k$, whereas QUP-linear allows them as long as the $k$th root converges. The sequence $\xi_k = Q_1^k k^a$ ($a>0$) illustrates this: it is Q-linear (with limit $Q_1=C_\psi$) but not UP-linear because the relative factor $\xi_k/Q_1^k = k^a$ is unbounded.

Second, \Cref{cor:qup-linear-implies-q-linear} requires that the Q-limit $Q_1$ exists. QUP-linear convergence ($\lim \xi_k^{1/k} = C_\psi$ exists) does \emph{not} imply Q-linear convergence ($\lim \xi_{k+1}/\xi_k = Q_1$ exists) or even the weaker EQ-linear bounds ($\alpha_1 \xi_k \le \xi_{k+1} \le \beta_1 \xi_k$). For example, the sequence $\xi_k = (C_\psi)^k k^{(-1)^k}$ is QUP-linear with P-base $C_\psi$, but it is neither Q-linear nor EQ-linear because the ratio $\xi_{k+1}/\xi_k$ oscillates unboundedly. Furthermore, we cannot replace Q-linear with the weaker EQ-linear \eqref{eq:exact-q-order} in \Cref{thm:q-linear-implies-qup-linear} or \Cref{cor:qup-linear-implies-q-linear}, since EQ-linear only ensures $\alpha_1 \le \liminf \xi_k^{1/k} \le \limsup \xi_k^{1/k} \le \beta_1$, which does not guarantee the existence of the limit $C_\psi = \lim \xi_k^{1/k}$ for QUP-linear convergence.

\subsubsection{Q-Superlinear and UP-Superlinear Convergence}
Unlike the linear case, Q-superlinear convergence implies the stronger UP-superlinear condition with the corresponding exponential power function, ensuring greater regularity.

\begin{theorem}[Q-Superlinear Implies UP-Superlinear]
\label{thm:q-superlinar-implies-up}
If a sequence converges Q-superlinearly with order $q>1$ and factor $Q_q \in (0, \infty)$, i.e.,
$\lim_{k\to\infty}\frac{\|\boldsymbol{x}_{k+1}-\boldsymbol{x}_*\|}{\|\boldsymbol{x}_k-\boldsymbol{x}_*\|^q} = Q_q,$
then it converges UP-superlinearly with P-function $\psi(k)=q^k$ (UP-order-$q$). The P-base $C_\psi$ may depend on $\boldsymbol{x}_0$.
\end{theorem}
\begin{proof}
Let $\xi_k = \|\boldsymbol{x}_k-\boldsymbol{x}_*\|$ and $f(k) = -\ln \xi_k$. The Q-superlinear condition $\xi_{k+1}/\xi_k^q \to Q_q$ gives the recurrence
$f(k+1) = q f(k) + d(k)$, where $d(k) \to -\ln Q_q$
with $d(k)$ bounded. Since $q>1$, dividing by $q^{k+1}$ and summing yields convergence of $f(k)/q^k$ to a limit $s$ (as $\sum d(k)/q^{k+1}$ converges absolutely).

Writing $f(k) = sq^k + R_k$ with $R_k = -q^k \sum_{j=k}^\infty d(j)/q^{j+1}$, we get $R_k = \mathcal{O}(1)$ since $d(j)$ is bounded. Thus $f(k) = sq^k + \mathcal{O}(1)$. As $f(k) \to \infty$, we need $s > 0$. Exponentiating:
\[ \xi_k = e^{-f(k)} = e^{-sq^k - \mathcal{O}(1)} = e^{-sq^k} e^{\mathcal{O}(1)} = \Theta(e^{-sq^k}). \]

With P-base $C_\psi = e^{-s} \in (0,1)$, we have $\xi_k = \Theta(C_\psi^{q^k})$, establishing UP-superlinear convergence with P-function $\psi(k)=q^k$.
\end{proof}

\Cref{thm:q-superlinar-implies-up} establishes that Q-superlinear convergence implies the stronger UP-superlinear condition ($\xi_k = \Theta(C_\psi^{q^k})$), which in turn clearly implies QUP-superlinear convergence. However, QUP-superlinear does not imply UP or Q-superlinear convergence. The sequence $\xi_k = C^{q^k} k^a$ ($a\neq 0$) demonstrates this: it is QUP-superlinear (its $q^k$th root limit is $C$) but fails to be UP-superlinear (and thus not Q-superlinear) because $\xi_k / C^{q^k} = k^a$ violates the required $\Theta(1)$ bounds.

Furthermore, a key equivalence specific to the superlinear case ($q>1$) exists between UP-order and EQ-order: UP-superlinear convergence is equivalent to EQ-superlinear convergence \eqref{eq:exact-q-order}. This holds because the proof of \Cref{thm:q-superlinar-implies-up} relies only on the boundedness of the ratio $\xi_{k+1}/\xi_k^q$ (showing EQ is sufficient for UP), while conversely, UP-superlinear convergence inherently ensures this ratio is bounded (showing UP implies EQ). Notably, this equivalence does not extend to linear convergence ($q=1$), as noted earlier.

\subsubsection{Hierarchy of P-, Q-, and R-Orders}
\label{subsec:p-q-hierarchy}
Before concluding this section, \Cref{fig:p-q-hierarchy} visualizes the hierarchy for linear and exponential rates, where P-, Q-, and R-orders apply. R-order represents the broadest class \cite[Def.~9.2.5]{ortega1970iterative} but lacks precision, particularly for order 1 (cf. \Cref{rem:r-order-ambiguity}). As depicted, P-order and QUP-order refine R-order while encompassing Q-order. For $q=1$, UP-linear only partially overlaps Q-linear; however, for $q>1$ (i.e., QUP-exponential order), UP-order is equivalent to EQ-order and contains Q-order. 

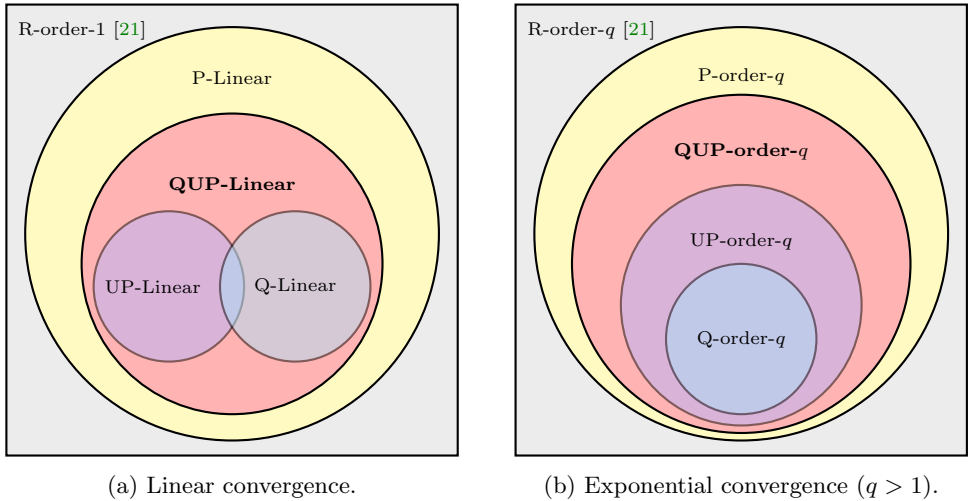
\begin{figure}
\begin{subfigure}[b]{0.48\textwidth}
\centering
\begin{tikzpicture}[
    every node/.style={font=\scriptsize, align=center}, 
    set/.style={circle, minimum size=2cm, draw=black, thick},
    subset/.style={set, fill=#1},
  ]

  \node[rectangle, minimum width=6cm, minimum height=6cm, draw=black, thick, fill=lightgray!30] (universe) at (0,-0.75) {}; 
  \node[above left=-0.6cm and -2cm of universe] {R-order-1 \cite{ortega1970iterative}}; 

  \node[subset=yellow!30, minimum width=5.5cm] (p_linear) at (0, -0.8) {};
  \node[above=-0.9cm of p_linear] {P-Linear};

  \node[subset=red!30, minimum width=4cm] (qup_linear) at (0, -1.2) {};
  \node[above=-1.2cm of qup_linear] {\textbf{QUP-Linear}};

  \node[subset=blue!30, minimum width=2cm, opacity=0.5, text opacity=1, text width=1.7cm, align=left] at (-0.84, -1.5) {UP-Linear};

  \node[subset=cyan!30, minimum width=2cm, opacity=0.5, text opacity=1] at (0.84, -1.5) {Q-Linear};
\end{tikzpicture}
\caption{Linear convergence.\label{fig:p-linear}}
\end{subfigure}
\hfill
\begin{subfigure}[b]{0.48\textwidth}
\centering
\begin{tikzpicture}[
    every node/.style={font=\scriptsize, align=center}, 
    set/.style={circle, minimum size=2cm, draw=black, thick},
    subset/.style={set, fill=#1},
  ]

  \node[rectangle, minimum width=6cm, minimum height=6cm, draw=black, thick, fill=lightgray!30] (universe) at (0,-0.75) {}; 
  \node[above left=-0.6cm and -2cm of universe] {R-order-$q$ \cite{ortega1970iterative}};

  \node[subset=yellow!30, minimum width=5.5cm] (p_superlinear) at (0, -0.8) {};
  \node[above=-0.9cm of p_superlinear] {P-order-$q$};

  \node[subset=red!30, minimum width=4.5cm] (banach) at (0, -1.2) {};
  \node[above=-1cm of banach] {\textbf{QUP-order-$q$}};

  \node[subset=blue!30, opacity=0.5, minimum width=3.2cm, text opacity=1] (up_superlinear) at (0, -1.75) {};
  \node[above=-1cm of up_superlinear] {UP-order-$q$};

  \node[subset=cyan!30, minimum width=1cm, opacity=0.5, text opacity=1] at (0, -2.2) {Q-order-$q$};
\end{tikzpicture}
\caption{Exponential convergence ($q>1$).\label{fig:p-superlinear}}
\end{subfigure}
\caption{Hierarchy of linear ($\psi=k$) and exponential ($\psi=q^k$, $q>1$) convergence rates. Not to scale (especially for the substantial UP/Q-linear overlap).}
\label{fig:p-q-hierarchy}
\end{figure}

The fundamental reason underneath this hierarchy is the different requirements of continuity assumptions of the derivatives for the different concepts that are associated with different remainder forms of Taylor's theorem (see, e.g., \cite{rudin1976principles} for different forms). In particular, the Peano form of the remainder offers qualitative bounds ($o(h^K)$) under the continuity of the $K$th derivative, which is sufficient to establish QUP-order (see, e.g., \Cref{thm:linear-superlinear}). The integral form provides a more precise bound ($\mathcal{O}(h^{K+1})$) under stronger continuity conditions (the absolute continuity of the $K$th derivative), and it is sufficient for establishing UP-order~(see \Cref{rem:high-order}). This is consistent with the fact that UP-order implies QUP-order. In contrast, proving Q-superlinear up to order $(K+1)$ \emph{requires} $C^{K+1}$ continuity (see, e.g., \cite[Property 6.4]{quarteroni2007numerical}), a condition stronger than the requirement of the strictest form of Taylor's theorem (the Lagrange form), which only requires the existence of the $(K+1)$th derivative. This is consistent with the fact that Q-superlinear implies UP-superlinear, but not vice versa. This connection illustrates the new insights that P-order framework brings to the convergence analysis of iterative methods, by mapping naturally to different remainder forms of Taylor's theorem.

\section{Refined Convergence Analysis of Fixed-Point Iterations}
\label{sec:fixed-point}

We now apply QUP-order to develop a refined analysis the fixed-point iteration
\begin{equation}\label{eq:fixed-point}
\boldsymbol{x}_{k+1} = \boldsymbol{g}(\boldsymbol{x}_k),
\end{equation}
where $\boldsymbol{g}\colon \mathbb{R}^n \to \mathbb{R}^n$ is continuously differentiable near the fixed point $\boldsymbol{x}_*$ (so that $\boldsymbol{x}_* = \boldsymbol{g}(\boldsymbol{x}_*)$).
We first refine a classical result using QUP-order for linear and superlinear convergence by relaxing the continuity conditions. Then, we present new results on sublinear fractional-power convergence.

\subsection{Linear and Superlinear with Weaker Continuity}
\label{sec:linear-superlinear-fp}

The classical Q-linear analysis requires $\lim_{k\to\infty} \|\xi_{k+1}\|/\|\xi_k\| = Q_1$, a limit that may not exist in standard $\ell^p$ norms even when $\rho(\boldsymbol{J})<1$ ($\boldsymbol{J}=\boldsymbol{J}_g(\boldsymbol{x}_*)$). Some textbooks (e.g., \cite{stoer2002introduction}) assume the existence of such norm, while others (e.g., \cite{quarteroni2007numerical}) establish the connection between $\rho(\boldsymbol{J})$ and a subordinate norm via the Rota--Strang theorem~\cite{rota1960note}. R-order analysis, while norm-independent, often establishes rates by comparison with Q-linear sequences \cite{ortega1970iterative}. More importantly, proving Q-linear requires $C^1$ continuity of $\boldsymbol{g}$. In contrast, our analysis differs in three ways: First, it directly yields the norm-independent rate $\lim_{k\to\infty}\|\boldsymbol{\xi}_k\|^{1/k} = \rho(\boldsymbol{J})$. Second, it assumes differentiability instead of $C^1$ continuity of $\boldsymbol{g}$. Third, it proves a sharp rate under a general-position assumption (where the component of the initial error in the dominant eigenspace of $\boldsymbol{J}$ is nonzero), instead of a lower bound. These differences of continuity requirement provide new insights into fixed-point iterations.

\begin{thm}[QUP-Linear and Superlinear]
\label{thm:linear-superlinear}
Let $\boldsymbol{g}$ be differentiable in a neighborhood of a fixed point $\boldsymbol{x}_*=\boldsymbol{g}(\boldsymbol{x}_*)$, with Jacobian $\boldsymbol{J}=\boldsymbol{J}_g(\boldsymbol{x}_*)$. Let $\mathcal{E}_\rho$ be the generalized eigenspace of $\boldsymbol{J}$ for eigenvalues with magnitude $\rho(\boldsymbol{J})$, and $\boldsymbol{P}_\rho$ the corresponding spectral projector. Assume the initial error $\boldsymbol{\xi}_0=\boldsymbol{x}_0-\boldsymbol{x}_*$ satisfies the \emph{general position} condition $\boldsymbol{P}_\rho\,\boldsymbol{\xi}_0\neq \boldsymbol{0}$. If $0 \le \rho(\boldsymbol{J}) < 1$, then for $\boldsymbol{x}_0$ sufficiently close to $\boldsymbol{x}_*$, the iteration $\boldsymbol{x}_{k+1}=\boldsymbol{g}(\boldsymbol{x}_k)$ converges to $\boldsymbol{x}_*$ such that
\[ \lim_{k\to\infty}\|\boldsymbol{x}_k-\boldsymbol{x}_*\|^{1/k}=\rho(\boldsymbol{J}). \]
Specifically, convergence is QUP-linear if $0<\rho(\boldsymbol{J})<1$ and QUP-superlinear if $\rho(\boldsymbol{J})=0$.
\end{thm}

\begin{proof}
Let $\boldsymbol{\xi}_k=\boldsymbol{x}_k-\boldsymbol{x}_*$. By Taylor's theorem in Peano's form, $\boldsymbol{\xi}_{k+1} = \boldsymbol{J}\boldsymbol{\xi}_k + \boldsymbol{r}(\boldsymbol{\xi}_k)$, where $\|\boldsymbol{r}(\boldsymbol{\xi}_k)\| = o(\|\boldsymbol{\xi}_k\|)$ for any differentiable function $\boldsymbol{g}$.

Consider the spectral decomposition $\boldsymbol{\xi}_k=\boldsymbol{P}_\rho\boldsymbol{\xi}_k+\boldsymbol{Q}_\rho\boldsymbol{\xi}_k$, where $\boldsymbol{P}_\rho$ projects onto the generalized eigenspace of $\boldsymbol{J}$ corresponding to eigenvalues with magnitude $\rho(\boldsymbol{J})$, and $\boldsymbol{Q}_\rho=\boldsymbol{I}-\boldsymbol{P}_\rho$. Since $\boldsymbol{P}_\rho$ and $\boldsymbol{J}$ commute, and $\rho(\boldsymbol{Q}_\rho\boldsymbol{J})<\rho(\boldsymbol{J})$ (when $\rho(\boldsymbol{J})>0$), the dominant dynamics are governed by $\boldsymbol{P}_\rho\boldsymbol{\xi}_k$.

The general position condition $\boldsymbol{P}_\rho\,\boldsymbol{\xi}_0\neq \boldsymbol{0}$ ensures $\boldsymbol{\xi}_k$ does not lie entirely in the faster-decaying subspace range$(\boldsymbol{Q}_\rho)$, thus $\|\boldsymbol{P}_\rho\boldsymbol{\xi}_k\|$ determines the asymptotic behavior of $\|\boldsymbol{\xi}_k\|$. Using induction on the Taylor expansion, for any $\epsilon>0$, there exist constants $C_1,C_2>0$ such that for all sufficiently large $k$,
\begin{equation}\label{eq:xi-k-inequality}
C_1\,\bigl(\rho(\boldsymbol{J})-\epsilon\bigr)^k \le \|\boldsymbol{\xi}_k\| \le C_2\,\bigl(\rho(\boldsymbol{J})+\epsilon\bigr)^k.
\end{equation}
Taking the $k$th root and letting $\epsilon\to 0$ yields $\lim_{k\to\infty}\|\boldsymbol{\xi}_k\|^{1/k}=\rho(\boldsymbol{J})$. The QUP conclusions follow directly from \Cref{def:p-quasi-uniform-order} based on the value of $\rho(\boldsymbol{J})$.
\end{proof}

In \Cref{thm:linear-superlinear}, Peano's form of Taylor's remainder leads to \eqref{eq:xi-k-inequality}, so we can only prove QUP-linear but not UP-linear or Q-linear, since the latter two require stronger continuity conditions. This result demonstrates a key advantage of QUP-order, as we emphasized in \Cref{subsec:p-q-hierarchy}. Another important aspect of the theorem is the general-position assumption. Without the assumption, the sequence may converge superlinearly for $\boldsymbol{x}_0$ in a lower-dimensional (i.e., measure-zero) submanifold of $B(\boldsymbol{x}_*, \delta)$ when $\boldsymbol{J_g}(\boldsymbol{x}_*)$ is singular but $\rho(\boldsymbol{J_g})>0$.

\begin{rem}[Higher-Order Convergence]\label{rem:high-order}
Similar to \Cref{thm:linear-superlinear}, we can establish QUP-order-$q$ convergence with $q>1$ under $C^{q-1,1}$ continuity and non-vanishing one-sided $q$th directional derivatives at $\boldsymbol{x}_*$ (sufficient condition for the integral remainder form of Taylor's theorem, instead of $C^q$ continuity). Furthermore, we can establish UP-order-$q$ if the $q$th derivative exists (i.e., the precise condition for the mean-value remainder form of Taylor's theorem). These distinction further reinforce the advantage of having the hierarchy of P-order, QUP-order, and UP-order, as we emphasized in \Cref{subsec:p-q-hierarchy}. We omit the theorem and proof due to space limitations. (See the supplemental material.)
\end{rem}

\subsection{A Sufficient Condition for Fractional-Power QUP-Order}
\label{sec:sublinear-fractional}

While $\rho(\boldsymbol{J_g}(\boldsymbol{x}_*)) = 1$ typically indicates sublinear convergence, it is insufficient to determine the rate. QUP-order enables a refined analysis, for instance, when the error reduction involves specific asymptotic behavior. The following theorem establishes a sufficient condition for QUP-fractional-power convergence, characterized by a logarithmic modulation of the error reduction projected onto the dominant eigenspace.

\begin{thm}[Iteration Functions with Fractional-Power Convergence Rate]
\label{thm:fp-fractional-power}
Let $\boldsymbol{g}:\mathcal{U}\to\mathbb{R}^n$ be continuously differentiable in an open neighborhood $\mathcal{U}$ of a fixed point $\boldsymbol{x}_*=\boldsymbol{g}(\boldsymbol{x}_*)$. Assume:
\begin{enumerate}
\item $\rho(\boldsymbol{J_g}(\boldsymbol{x}))<1$ for $\boldsymbol{x}\in \mathcal{U}\setminus\{\boldsymbol{x}_*\}$, and $\rho(\boldsymbol{J_g}(\boldsymbol{x}_*))=1$.
\item Let $\mathcal{E}_\rho$ be the generalized eigenspace for eigenvalues of $\boldsymbol{J_g}(\boldsymbol{x}_*)$ with unit modulus, and $\boldsymbol{P}_\rho$ the corresponding projector. There exist $s>0$ and $C_0$ such that for some $\ell^p$ norm,
\begin{equation}\label{eq:sublinear-condition}
\lim_{x\rightarrow x_*}\frac{\|\boldsymbol{P}_\rho (\boldsymbol{x}-\boldsymbol{g}(\boldsymbol{x}))\|}{
\|\boldsymbol{P}_\rho (\boldsymbol{x}-\boldsymbol{x}_*)\|}\Bigl(-\ln \|\boldsymbol{P}_\rho (\boldsymbol{x}-\boldsymbol{x}_*)\|\Bigr)^{s}
=C_0.
\end{equation}
\end{enumerate}
Then, for any $\boldsymbol{x}_0$ sufficiently close to $\boldsymbol{x}_*$ in general position (i.e., $\boldsymbol{P}_\rho(\boldsymbol{x}_0-\boldsymbol{x}_*)\neq\boldsymbol{0}$), the sequence $\boldsymbol{x}_{k+1}=\boldsymbol{g}(\boldsymbol{x}_k)$ converges to $\boldsymbol{x}_*$ with QUP-fractional-power P-order $1/(s+1)$. Specifically, $\psi(k) = k^{1/(s+1)}$, and
\[
\lim_{k\to\infty}\|\boldsymbol{x}_k-\boldsymbol{x}_*\|^{1/\psi(k)}= \lim_{k\to\infty}\|\boldsymbol{x}_k-\boldsymbol{x}_*\|^{k^{-1/(s+1)}}=C_{\psi}
\]
holds in all $\ell^p$ norms, where $ C_{\psi} = \exp(-(C_0(s+1))^{1/(s+1)})$.
\end{thm}

\begin{proof}
Let $\boldsymbol{\xi}_
k = \boldsymbol{x}_k - \boldsymbol{x}_* = \boldsymbol{P}_\rho\boldsymbol{\xi}_k + \boldsymbol{Q}_\rho\boldsymbol{\xi}_k$, where $\boldsymbol{P}_\rho$ projects onto the unit-modulus generalized eigenspace\footnote{The generalized eigenspace of a matrix $\boldsymbol{A}$ corresponding to eigenvalue $\lambda$ is the set of all vectors $\boldsymbol{v}$ such that $(\boldsymbol{A}-\lambda\boldsymbol{I})^k\boldsymbol{v}=\boldsymbol{0}$ for some positive integer $k$. See, e.g., \cite[Ch. 6]{horn2013matrix}.} of $\boldsymbol{J_g}(\boldsymbol{x}_*)$. Since $\rho(\boldsymbol{J_g}(\boldsymbol{x}_*) |_{\text{range}(\boldsymbol{Q}_\rho)}) < 1$, we have $\|\boldsymbol{Q}_\rho\boldsymbol{\xi}_k\| = o(\gamma^k)$ for some $\gamma<1$. Let $u_k = \|\boldsymbol{P}_\rho\boldsymbol{\xi}_k\|$. From $\boldsymbol{\xi}_{k+1} = \boldsymbol{\xi}_k - (\boldsymbol{x}_k - \boldsymbol{g}(\boldsymbol{x}_k))$ and Condition \eqref{eq:sublinear-condition}, we obtain
\[ u_{k+1} = u_k\left(1 - C_0(-\ln u_k)^{-s} + o((-\ln u_k)^{-s})\right). \]

Setting $v_k = -\ln u_k$ and taking logarithms yields $v_{k+1} = v_k + C_0 v_k^{-s} + o(v_k^{-s})$. This standard recurrence implies $v_k \sim ((s+1)C_0 k)^{1/(s+1)}$ (see \cite[Prop. 6.3.8]{tao2016analysis}). Since $u_k = e^{-v_k} = e^{-\Theta(k^{1/(s+1)})}$ decays sub-exponentially, it dominates $\|\boldsymbol{Q}_\rho\boldsymbol{\xi}_k\|$, giving $\|\boldsymbol{\xi}_k\| \sim u_k$. The QUP-order limit is therefore
\[
\lim_{k\to\infty}\|\boldsymbol{\xi}_k\|^{k^{-1/(s+1)}} = \lim_{k\to\infty}e^{-v_k/k^{1/(s+1)}} = e^{-((s+1)C_0)^{1/(s+1)}} = C_{\psi},
\]
confirming QUP-order $1/(s+1)$ with $\psi(k)=k^{1/(s+1)}$ and $C_\psi \in (0,1)$.
\end{proof}

We will apply this theorem to derive fractional-power convergence rates for Newton's method and gradient descent in \Cref{ex:newton-fractional} and \Cref{ex:gd_designed_fractional}, respectively, which involve verifying the two conditions in \Cref{thm:fp-fractional-power}.

\section{Example New Sublinear Rates of Fixed-Point-Type Methods}
\label{sec:examples-sublinear}

Having established the P-order framework and analyzed general fixed-point iterations, we now apply it to demonstrate novel sublinear convergence behaviors for two prominent iterative methods: \emph{Newton's method for root-finding} and \emph{gradient descent for optimization}. Both can be viewed as specific instances of the fixed-point iteration $\boldsymbol{x}_{k+1} = \boldsymbol{g}(\boldsymbol{x}_k)$. We present examples exhibiting fractional-power, linearithmic, and anti-linearithmic QUP-order rates, highlighting the framework's ability to capture convergence patterns beyond those described by classical Q-order or R-order analysis.

\subsection{Newton's Method for Nonlinear Equations}
\label{subsec:newton_examples}

We first consider Newton's method (a.k.a. the Newton--Raphson method) applied to scalar nonlinear equations $f(x)=0$. The iteration is defined by the fixed-point function
\begin{equation}\label{eq:newton}
x_{k+1} = g(x_k), \quad \text{where} \quad g(x) = x - \frac{f(x)}{f'(x)}.
\end{equation}
For systems $\boldsymbol{f}(\boldsymbol{x})=\boldsymbol{0}$, the iteration function is $ \boldsymbol{g}(\boldsymbol{x}) = \boldsymbol{x} - \boldsymbol{J_f}(\boldsymbol{x})^{-1}\boldsymbol{f}(\boldsymbol{x}) $. These fixed-point iterations can be analyzed using Theorems~\ref{thm:linear-superlinear} and \ref{thm:fp-fractional-power}, but we also consider new results for linearithmic and anti-linearithmic convergence rates.

Classical analysis \cite{kantorovich1948functional, traub1964iterative, ortega1970iterative} establishes quadratic (or higher order) convergence for simple roots ($f'(x_*) \neq 0$). For singular roots ($f'(x_*) = 0$ or $\boldsymbol{J_f}(\boldsymbol{x}_*)$ singular), behavior is more complex \cite{rall1966multiple, reddien1978newton, decker1980newton, decker1983convergence, griewank1985solving}. Notably, for analytic functions\footnote{\label{ft:analytic}A function $f\colon \mathbb{R}^n \to \mathbb{R}^m$ is said to be \emph{analytic} at a point $\boldsymbol{x}_0$ if it can be represented by a convergent power series in a neighborhood of $\boldsymbol{x}_0$. Analytic functions are infinitely differentiable, but the opposite is not necessarily true. A simple example of a function that is infinitely differentiable but not analytic is the function $f(x) = e^{-1/x^2}$ for $x\neq 0$ and $f(x) =0$ for $x = 0$.} with multiple roots of finite multiplicity, convergence is linear \cite{decker1980newton}. While sublinear convergence was noted as possible \cite{griewank1985solving}, detailed examples and analysis of specific sublinear rates for Newton's method appear absent from the literature.

We present examples where Newton's method exhibits diverse, previously unreported sublinear QUP-order rates, specifically fractional-power, linearithmic, and anti-linearithmic. Logarithmic convergence is relatively straightforward (e.g., for the function in \Cref{ft:analytic}) and not detailed here. These examples necessarily involve \emph{non-analytic} functions; this is a corollary of the results in \cite{decker1980newton}, where Decker and Kelley showed linear convergence for analytic functions with roots of finite multiplicity along with the trivial convergence for analytic functions vanishing to infinite order. The non-analytic examples highlight the sensitivity of Newton's method to local function properties near singularities. We focus on scalar cases for simplicity, though the concepts extend to multivariate systems.

\subsubsection{Fractional-Power Convergence}

We first consider cases where Newton's method converges with a sublinear fractional-power rate. Such rates can occur when the derivative vanishes at the root ($f'(x_*) = 0$) but relatively slowly, corresponding to the conditions of \Cref{thm:fp-fractional-power}.

\begin{exam}[Fractional-Power Convergence of Newton's Method]
\label{ex:newton-fractional}
Consider the scalar equation $f(x)=0$ with $f(x)$ defined by
\begin{equation}\label{eq:fractional_power_newton}
f(x)=
\begin{cases}
\exp\Bigl(-c\,\bigl(-\ln|x-\alpha|\bigr)^{1/r}\Bigr), & x\neq\alpha,\\[1ex]
0, & x=\alpha,
\end{cases}
\end{equation}
where $c>0$, $\alpha\in\mathbb{R}$, and $0<r<1$. This function has a root $f(\alpha)=0$, is infinitely differentiable, but not analytic at $\alpha$.
To apply Newton's method \eqref{eq:newton}, we need $f'(x)$. Let $L(x) = -\ln|x-\alpha|$. A calculation using the chain rule yields
\[ f'(x) = f(x)\, \frac{c}{r} \bigl(L(x)\bigr)^{\frac{1}{r}-1}\, \frac{1}{x-\alpha}. \]
The Newton iteration function is $g(x) = x - f(x)/f'(x)$. Substituting $f'(x)$, we obtain
\[
g(x) = x - \frac{f(x)}{f(x) \cdot \frac{c}{r} (L(x))^{\frac{1}{r}-1} \frac{1}{x-\alpha}}
= x - (x-\alpha)\frac{r}{c} \bigl(L(x)\bigr)^{1-\frac{1}{r}}.
\]
To check Condition \eqref{eq:sublinear-condition} of \Cref{thm:fp-fractional-power}, we examine the term involving $x-g(x)$. From the expression for $g(x)$, we obtain
\[
x - g(x) = (x-\alpha)\frac{r}{c} (-\ln|x-\alpha|)^{1-1/r}.
\]
Let $\epsilon = |x-\alpha|$. We then form the required limit expression
\[
\frac{x-g(x)}{x-\alpha} (-\ln\epsilon)^{1/r-1} = \left( \frac{r}{c} (-\ln\epsilon)^{1-1/r} \right) (-\ln\epsilon)^{1/r-1} = \frac{r}{c}.
\]
Thus, Condition \eqref{eq:sublinear-condition} holds (noting $P_\rho=1$ in the scalar case) with constants $s = 1/r - 1 > 0$ and $C_0 = r/c$. Therefore, \Cref{thm:fp-fractional-power} guarantees convergence with QUP-order $1/(s+1)=r$, corresponding to the power function $\psi(k)=k^r$.
\end{exam}

\Cref{fig:fracx} illustrates this convergence numerically for $f(x)=\exp\{-(-\ln x)^{1/r}\}$ (\eqref{eq:fractional_power_newton} with $\alpha=0, c=1$) using $r \in \{0.25, 0.5, 0.75\}$ and $x_0=0.01$. For comparison, linear ($f(x)=x^2$) and quadratic ($f(x)=x-x^2$) convergence are also shown. The observed error decay rates align with the theoretical fractional-power orders derived in \Cref{ex:newton-fractional}, demonstrating convergence distinctly slower than linear.

\begin{figure}[htb]
\begin{minipage}[b]{0.49\columnwidth}
    \centering
    \includegraphics[width=\columnwidth]{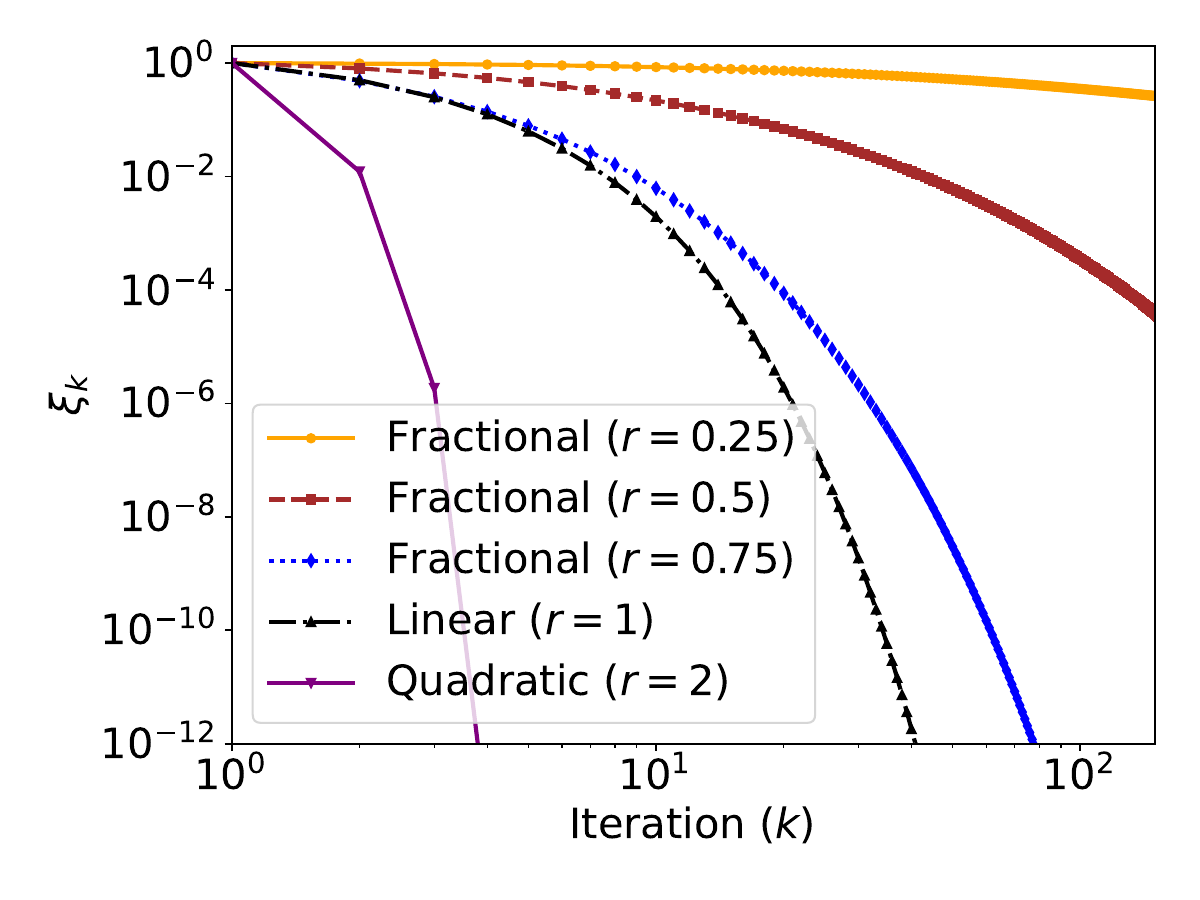}
    \captionof{figure}{Examples of sublinear fractional-power rates of Newton's method.}
    \label{fig:fracx}
  \end{minipage}
  \hfill
  \begin{minipage}[b]{0.49\columnwidth}
    \centering
    \includegraphics[width=\columnwidth]{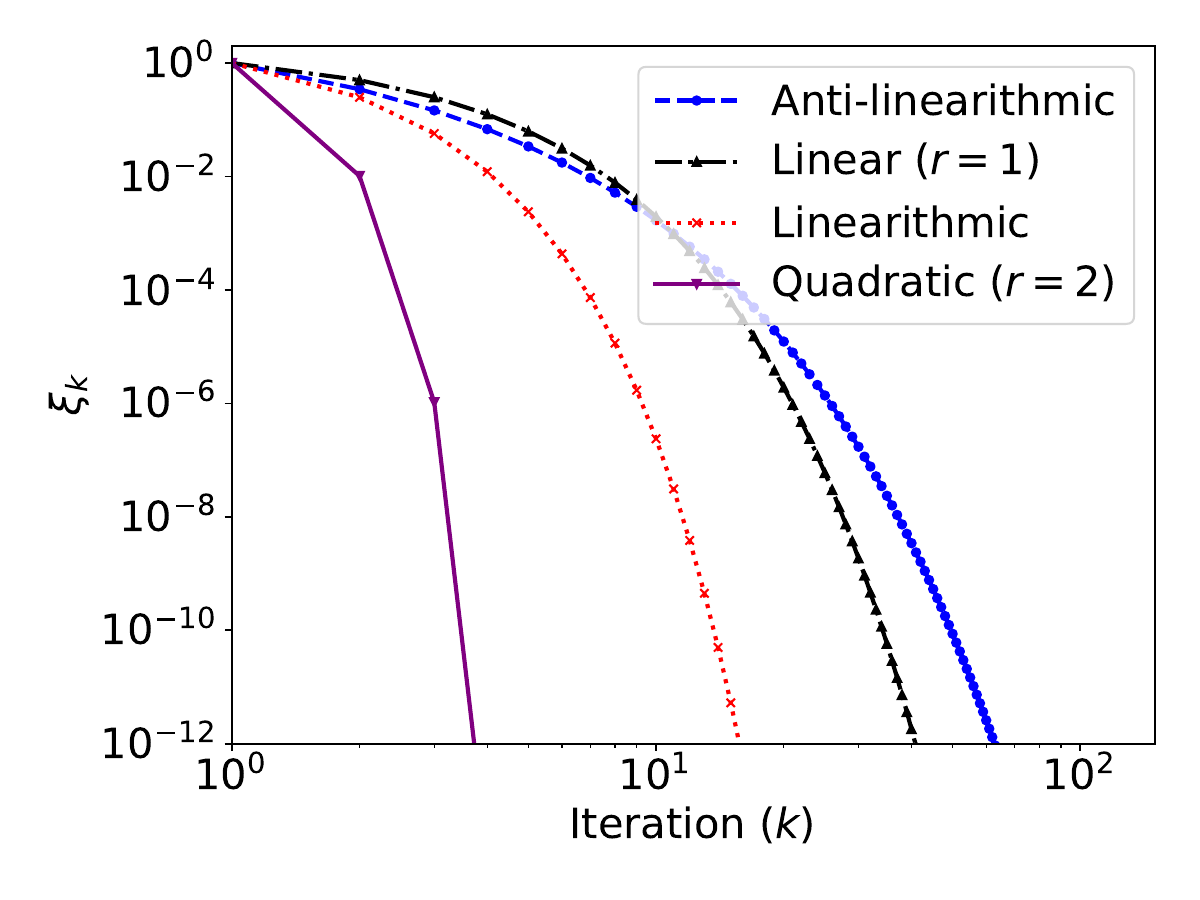}
    \captionof{figure}{Examples of (anti-)linearithmic rates of Newton's method.}
    \label{fig:lithx}
  \end{minipage}
\end{figure}

\subsubsection{Linearithmic and Anti-Linearithmic Convergence}

We now present examples demonstrating convergence rates involving logarithmic factors combined with linear or sublinear terms in the exponent base $k$, namely linearithmic ($\psi(k)=k\ln k$) and anti-linearithmic ($\psi(k)=k/\ln k$) rates, illustrated in \Cref{fig:lithx}. These rates, not previously explored for Newton's method, highlight the refined classification provided by QUP-order. They occupy intermediate positions: linearithmic is superlinear but slower than any exponential P-order-$r$ ($r>1$), while anti-linearithmic is sublinear but faster than any fractional-power P-order-$r$ ($r<1$). Understanding these nuances clarifies the limitations of frameworks like R-order (cf. \Cref{rem:r-order-ambiguity}).

\begin{exam}[Linearithmic Convergence]
\label{ex:linearithmic}
Define function $f(x)$ with root at $x=\alpha$ as
\begin{equation}\label{eq:lith_func_modified}
f(x)=
\begin{cases}
|x-\alpha|\,
\exp\!\Bigl(-\frac{1}{2}\Bigl(\ln(-\ln|x-\alpha|)\Bigr)^2\Bigr), &
x\neq \alpha,\\[1ex]
0, & x=\alpha.
\end{cases}
\end{equation}
This function is infinitely differentiable for $x\neq\alpha$ but not analytic at $x=\alpha$.
Let $\epsilon = |x-\alpha|$ and $L = -\ln\epsilon$. Using the auxiliary function $t(x) = \ln f(x) = -L - \frac{1}{2}(\ln L)^2$, we compute the Newton iteration function $g(x)=x-f(x)/f'(x) = x-1/t'(x)$ where
\[ t'(x) = \frac{1}{\epsilon}\left(1 + \frac{\ln L}{L}\right). \]

For the error ratio, we obtain
\[ \frac{\epsilon_{k+1}}{\epsilon_k} \approx \frac{\ln L}{L}. \]
This implies $k \sim L/\ln L$, giving the recurrence $\epsilon_{k+1} \sim \epsilon_k/k$. Solving asymptotically yields $\epsilon_k \sim \epsilon_0 / (k-1)!$, and applying Stirling's approximation gives $-\ln\epsilon_k \approx k\ln k - k$. Therefore,
\[ \epsilon_k = \Theta\Bigl( C^{k\ln k} \Bigr), \]
where $C=e^{-1} \in (0,1)$, establishing UP-linearithmic convergence with $\psi(k)=k\ln k$.
\end{exam}

\begin{exam}[Anti-Linearithmic Convergence]
\label{ex:anti-linearithmic}
Consider the function $f(x)$ for $\alpha \in \mathbb{R}$,
\begin{equation}\label{eq:anti_lith_func_modified}
f(x)=
\begin{cases}
|x-\alpha|^{\,\ln(-\ln|x-\alpha|)-1}, & x\neq \alpha,\\[1ex]
0, & x=\alpha.
\end{cases}
\end{equation}
This infinitely differentiable (but non-analytic) function has a root at $f(\alpha)=0$. Setting $\epsilon = |x-\alpha|$, $L' = -\ln\epsilon$, and $L = \ln L'$, we use $t(x) = \ln f(x) = -(L-1)L'$. Computing $t'(x) = \frac{\ln L'}{\epsilon}$ yields the Newton iteration error ratio
\[ \frac{\epsilon_{k+1}}{\epsilon_k} \approx \left| 1 - \frac{1}{\epsilon t'(x)} \right| \approx 1 - \frac{1}{\ln L'}. \]

Associating index $k$ via $k \sim L' / \ln L'$ gives $\epsilon_{k+1} \sim \epsilon_k (1 - 1/\ln k)$, which leads to $\ln \epsilon_{k+1} - \ln \epsilon_k \sim -1/\ln k$. The sum $\sum_{j=2}^{k-1} (1/\ln j)$ is asymptotically equivalent to the logarithmic integral $\text{li}(k) \sim k/\ln k$, yielding
\[ \epsilon_k \sim \epsilon_0 \exp\!\Bigl(-\frac{k}{\ln k}\Bigr) = \Theta\Bigl( C^{k/\ln k} \Bigr), \]
where $C=e^{-1} \in (0,1)$, establishing UP-anti-linearithmic convergence with $\psi(k)=k/\ln k$.
\end{exam}

\Cref{fig:lithx} provides numerical confirmation. Newton's method was applied to the functions from Examples~\ref{ex:linearithmic} and \ref{ex:anti-linearithmic} with $\alpha=0$. The observed convergence rates qualitatively match the theory: linearithmic convergence is slightly faster than linear but markedly slower than quadratic, while anti-linearithmic convergence is only slightly slower than linear.

\subsection{Gradient Descent for Optimization}
\label{subsec:gd_examples}

The QUP-order framework is effective for analyzing not only iterative solvers of nonlinear equations but also optimization methods. While adapting Newton's method examples from \Cref{subsec:newton_examples} to optimization is straightforward, gradient descent (GD), $\boldsymbol{x}_{k+1} = \boldsymbol{x}_k - \eta_k \boldsymbol{\nabla} F(\boldsymbol{x}_k)$, requires additional derivation. Classical GD analysis typically assumes Lipschitz-continuous gradients~\cite{boyd2004convex,nesterov2004introductory,nocedal2006numerical}, yielding linear or sublinear rates that typically correspond to QUP-linear and QUP-logarithmic, respectively. Recent work has explored weaker conditions~\cite{berger2020quality}, for which the convergence rates also correspond to QUP-logarithmic.

We demonstrate that GD can exhibit QUP-fractional-power convergence rates, which has not been previously reported. To this end, we design radially symmetric objective functions $F(\boldsymbol{x}) = \Phi(\|\boldsymbol{x}\|)$, with particular gradient behaviors near the minimum $\boldsymbol{x}_* = \boldsymbol{0}$, where $r = \|\boldsymbol{x}\|$. The gradient is $\boldsymbol{\nabla} F(\boldsymbol{x}) = \Phi'(r) \frac{\boldsymbol{x}}{r}$. Let $f(r) = \Phi'(r)$, and then $\boldsymbol{\nabla} F(\boldsymbol{x}) = f(r) \frac{\boldsymbol{x}}{r}$. The GD iteration with fixed step size $\eta$ becomes
\[
\boldsymbol{x}_{k+1} = \boldsymbol{x}_k - \eta f(r_k) \frac{\boldsymbol{x}_k}{r_k} = \boldsymbol{x}_k \left( 1 - \eta \frac{f(r_k)}{r_k} \right).
\]
Taking the Euclidean norm ($r_k = \|\boldsymbol{x}_k\|$), we obtain the scalar recurrence for the norm
\[
r_{k+1} = \|\boldsymbol{x}_{k+1}\| = r_k \left| 1 - \eta \frac{f(r_k)}{r_k} \right|.
\]
This scalar recurrence can be analyzed using the techniques from previous sections by choosing $f(r)$ appropriately.

\begin{exam}[Fractional-Power Convergence of GD]
\label{ex:gd_designed_fractional}
For $0 < r < 1$, define $f(\rho) = \rho(-\ln \rho)^{1-1/r}$ for $\rho \in (0, e^{-1})$. We construct the objective function $F(\boldsymbol{x}) = \int_0^{\|\boldsymbol{x}\|} f(\rho) d\rho$, expressible using the upper incomplete gamma function \cite[Sec.~8.2]{olver2010nist} as
\begin{equation}\label{eq:gd_frac_potential_F}
F(\boldsymbol{x}) = \left(\frac{1}{2}\right)^{2-1/r} \Gamma\left(2-\frac{1}{r}, -2\ln\|\boldsymbol{x}\|\right).
\end{equation}
With gradient $\boldsymbol{\nabla} F(\boldsymbol{x}) = f(r) \boldsymbol{x}/r$ (where $r = \|\boldsymbol{x}\|$), gradient descent using step size $\eta=1$ yields $r_{k+1} = r_k |1 - (-\ln r_k)^{1-1/r}|$.

Setting $v_k = -\ln r_k$ and taking logarithms, we obtain the recurrence $v_{k+1} - v_k \sim v_k^{1-1/r}$. With $s = 1/r - 1 > 0$, this yields $v_k \sim (1/r)^r k^r$. The QUP-order limit with $\psi(k)=k^r$ is
\begin{equation}\label{eq:gd_qup_limit}
\lim_{k\to\infty} (r_k)^{1/\psi(k)} = \lim_{k\to\infty} e^{-v_k/k^r} = e^{-(1/r)^r},
\end{equation}
confirming QUP-fractional-power order $r$ with P-base $C_\psi = e^{-(1/r)^r} \in (0,1)$.

\Cref{fig:gd_frac_plots} validates this analysis, showing error decay for $r \in \{0.25, 0.5, 0.75\}$ (\Cref{fig:gd_frac_error}) and the numerically estimated convergence exponents approaching their theoretical values (\Cref{fig:gd_frac_rate}).
\end{exam}

\begin{figure}[htb]
\centering
\subfloat[Error decay $\|\boldsymbol{x}_k\|$ vs. $k$. \label{fig:gd_frac_error}]{%
    \includegraphics[width=0.48\textwidth]{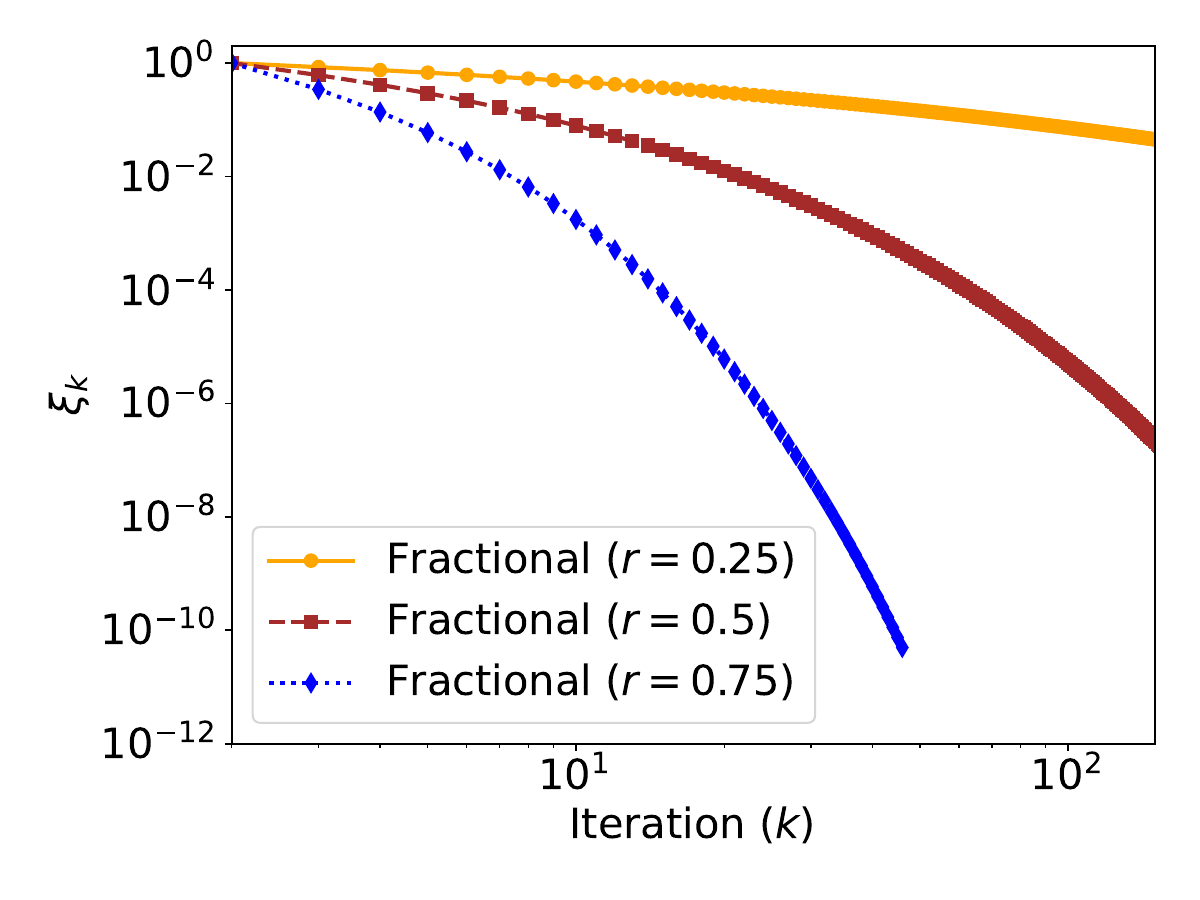}
}
\hfill
\subfloat[Estimated rate $r_k^{\text{est}}$ vs. $k$. \label{fig:gd_frac_rate}]{%
    \includegraphics[width=0.48\textwidth]{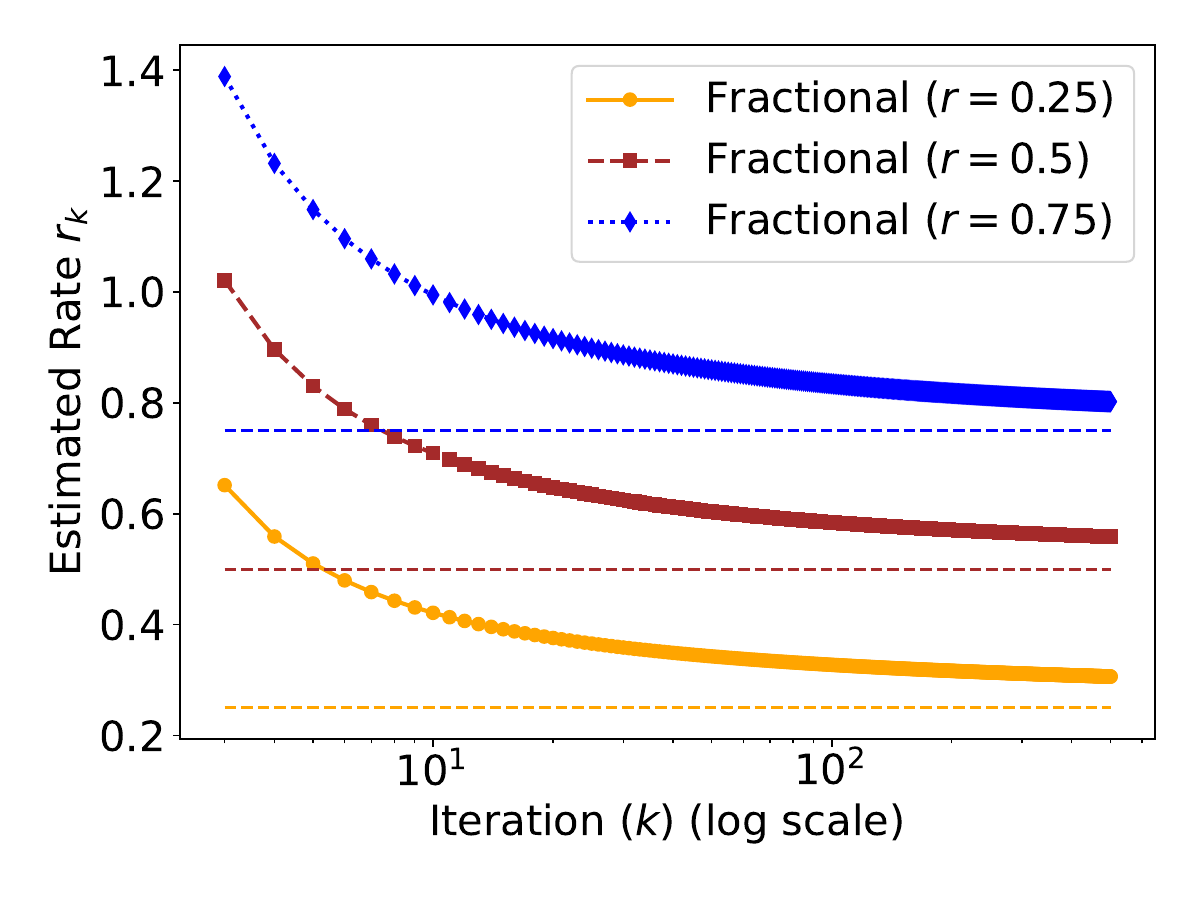}
}
\caption{Numerical illustration of QUP-fractional-power convergence ($r=0.25, 0.5, 0.75$) for gradient descent (\Cref{ex:gd_designed_fractional}).}
\label{fig:gd_frac_plots}
\end{figure}

In addition, we can construct an objective function whose gradient mimics the dynamics of the anti-linearithmic example for Newton's method (\Cref{ex:anti-linearithmic}), leading to UP-anti-linearithmic convergence for gradient descent. Its corresponding potential function $F(\boldsymbol{x}) = \int_{-\ln\|\boldsymbol{x}\|}^{\infty} (e^{-2t}/\ln t) dt$ lacks a representation in terms of standard special functions. We omit the detailed construction and analysis for this case.

\section{\texorpdfstring{$K$-Point Iterative Methods under H\"older Continuity}{K-Point Iterative Methods under H\"older Continuity}}
\label{sec:multipoint}

Previous sections demonstrated the utility of the P-order framework for analyzing methods like Newton's and gradient descent, particularly revealing sublinear behaviors. We now extend this analysis to a class of interpolation-based \emph{multipoint iterative methods}\footnote{These methods are also referred to as \emph{single-point methods with memory} in the classification by Traub~\cite{traub1964iterative}. Our terminology follows~\cite{tornheim1964convergence}, which predated~\cite{traub1964iterative}.} for solving nonlinear scalar equations $f(x)=0$. The secant method ($K=2$) and (inverse) Muller's method ($K=3$)~\cite{muller1956method} are well-known examples. While these specific low-order methods are foundational, modern high-performance root-finding algorithms often employ multipoint schemes with higher values of $K$ (e.g., $K=4$) based on higher-order interpolation. For instance, robust solvers like Brent's method~\cite{brent1973algorithms,alefeld1995algorithm}, often combine safe bracketing techniques with fast multipoint methods.

Classical convergence analysis for $K$-point iterative methods assumes $f \in C^K$ near the root $x_*$. For methods based on polynomial interpolation, such as inverse interpolation where $x_{k+1} = P_k(0)$ (using $k$ for iteration index) with $P_k(f(x_{k-i})) = x_{k-i}$ for $i=0, \dots, K-1$, this leads to order $q_K(1)$, the positive root of $q^K - q^{K-1} - \dots - q - 1 = 0$~\cite{tornheim1964convergence}, for which Tornheim assumed $C^{K}$ continuity and used the definition of ``order'' due to Wall~\cite{wall1956order}. We now derive a new result when this smoothness condition is relaxed to $f \in C^{K-1,\nu}$, meaning $f^{(K-1)}$ is H\"older continuous with exponent $\nu \in (0, 1]$. Since H\"older continuity is typically given as an inequality, we will use the big-$\mathcal{O}$ notation for UP-order, further demonstrating the flexibility of the P-order framework.

\begin{thm}[Convergence Rate of $K$-Point Methods for $C^{K-1,\nu}$ Functions]
\label{thm:Kpoint_holder_rate}
Let $f: I \to \mathbb{R}$ be defined on an open interval $I$. Assume $f$ has a simple root $x_* \in I$ (i.e., $f(x_*)=0, f'(x_*) \neq 0$) and that $f \in C^{K-1}(I)$ for some integer $K \ge 2$. Furthermore, if $f^{(K-1)}$ is H\"older continuous with exponent $\nu \in (0, 1]$ and constant $L_\nu > 0$ in a neighborhood of $x_*$, i.e.,
\begin{equation}\label{eq:holder_K_minus_1_thm}
 |f^{(K-1)}(x) - f^{(K-1)}(y)| \le L_\nu |x-y|^\nu \quad \text{for } x,y \text{ near } x_*.
\end{equation}
Let $\{x_k\}$ (using $k$ as iteration index) be a sequence generated by a $K$-point iterative method based on inverse polynomial interpolation. If the initial points $x_0, \dots, x_{K-1}$ are sufficiently close to $x_*$, then the sequence converges superlinearly to $x_*$ and the method converges with a UP-order of at least $q_K(\nu)$, i.e., $|\epsilon_k| = \mathcal{O}(A^{(q_K(\nu))^k})$ for some $A \in (0,1)$, where $q_K(\nu)$ is the unique positive root of the characteristic equation
\begin{equation} \label{eq:Kpoint_char_eq}
q^K = q^{K-1} + q^{K-2} + \dots + q + \nu.
\end{equation}
\end{thm}

\begin{proof}
Let $\{x_k\}$ be the sequence generated by the $K$-point iterative method based on inverse polynomial interpolation, $x_{k+1} = P_k(0)$, where $P_k$ interpolates $(y_{k-i}, x_{k-i})$ for $i=0,\dots,K-1$, with $y_j=f(x_j)$. Denote the error by $\epsilon_k = x_k - x_*$. For iterates sufficiently close to the simple root $x_*$ (where $f'(x_*) \neq 0$), the error satisfies
\begin{equation} \label{eq:error_relation}
|\epsilon_{k+1}| \le C |f[x_k, \dots, x_{k-K+1}, x_*]| \prod_{i=0}^{K-1} |f(x_{k-i})|
\end{equation}
for some constant $C > 0$. Since $f(x_*)=0$ and $f'(x_*) \neq 0$, we have $|f(x_{k-i})| \le M |\epsilon_{k-i}|$ for some $M>0$.

Since $f^{(K-1)}$ is Hölder continuous with exponent $\nu$, its modulus of continuity satisfies $\omega_1(f^{(K-1)}, t) \le L_\nu t^\nu$. This implies $\omega_K(f, t) \le C t^{K-1} \omega_1(f^{(K-1)}, t) \le C' t^{K-1+\nu}$. For divided differences, we have $|f[x_0, \dots, x_K]| \le \tilde{C} \omega_K(f, h) / h^K$ where $h = \operatorname{diam}\{x_0, \dots, x_K\}$. Combining these yields
\begin{equation} \label{eq:divdiff_bound_derived}
|f[x_k, \dots, x_{k-K+1}, x_*]| \le C_1 (\operatorname{diam}(I_k))^{\nu-1}
\end{equation}
where $I_k = \{x_k, \dots, x_{k-K+1}, x_*\}$.

We assume superlinear convergence (i.e., $|\epsilon_j| = o(|\epsilon_{j-1}|)$ as $j \to \infty$) and will verify this assumption retrospectively after establishing the convergence rate. With superlinear convergence, $|\epsilon_{k-K+1}|$ dominates in $I_k$, giving $\operatorname{diam}(I_k) = \Theta(|\epsilon_{k-K+1}|)$. Thus, $|f[x_k, \dots, x_{k-K+1}, x_*]| \le C_2 |\epsilon_{k-K+1}|^{\nu-1}$. Substituting into \eqref{eq:error_relation} yields
\begin{equation}
 |\epsilon_{k+1}|  \le C_{\text{new}} |\epsilon_k| |\epsilon_{k-1}| \dots |\epsilon_{k-K+1}|^{\nu} \label{eq:epsilon_recurrence_final}
\end{equation}
where $C_{\text{new}} = C C_2 M^K$.

To analyze this recurrence, consider the auxiliary sequence $\{z_k\}$ defined by $z_{k+1} = C_{\text{new}} z_k z_{k-1} \dots z_{k-K+1}^{\nu}$. This sequence converges as $z_k = \Theta(A^{(q_K(\nu))^k})$ where $q_K(\nu)$ is the positive root of \eqref{eq:Kpoint_char_eq} and $A \in (0,1)$. Since $q_K(\nu) > 1$, this confirms superlinear convergence. By choosing initial values $z_i \ge |\epsilon_i|$ and using induction with \eqref{eq:epsilon_recurrence_final}, we establish $|\epsilon_k| \le z_k$ for all $k$, thus
\[ |\epsilon_k| = \mathcal{O}\bigl(A^{(q_K(\nu))^k}\bigr). \]

This confirms our superlinear convergence assumption (justifying the diameter approximation) and establishes that the method has UP-order at least $q_K(\nu)$.
\end{proof}

\begin{rem}[Necessity of UQ-Order]
In the proof, the notion of moduli of continuity allows us to generalize this proof to derive characteristic equations for even more general classes of continuities, such as log-modulated H\"older continuity. In addition, we intentionally used a standard technique involving an auxiliary sequence $z_k$ to bound the error. This technique is related to proving R-order by constructing a bounding Q-order sequence~\cite{nocedal2006numerical,ortega1970iterative}. However, applying  Q-order analysis directly would prove fruitless here, because rigorous Q-order proofs for the secant method typically require $C^2$ continuity and do not apply for $C^{1,1}$, let alone $C^{1,\nu}$ for arbitrary $\nu\in(0,1]$. The analysis in \cite{hernandez2001secant} for the secant method ($K=2$) under H\"older-continuous derivatives required constructing specialized recursive sequences, a technique difficult to generalize to $K>2$. The P-order framework enabled a systematic analysis for all $K$ and $\nu$.
\end{rem}

We note that the derived rate $q_K(\nu)$ is expected to be sharp if the bound \eqref{eq:divdiff_bound_derived} is asymptotically tight (which occurs if the H\"older exponent $\nu$ is sharp near $x_*$), as demonstrated numerically below.

\begin{exam}[Secant Method ($K=2$)]
\label{ex:secant_holder}
For the secant method ($K=2$), the characteristic equation \eqref{eq:Kpoint_char_eq} simplifies to $q^2 = q + \nu$ with positive root $q_2(\nu) = \frac{1 + \sqrt{1+4\nu}}{2}$. By \Cref{thm:Kpoint_holder_rate}, if $f \in C^{1,\nu}$, the method has UP-order at least $q_2(\nu)$.
\begin{itemize}
 \item When $\nu=1$: $q_2(1) = \frac{1+\sqrt{5}}{2} \approx 1.618$ (classical results show this under $C^2$).
 \item As $\nu \to 0^+$: $q_2(\nu) \to 1$, degrading toward linear convergence.
 \item For $\nu=0.5$ ($f^{(1)}$ is $1/2$-Hölder): $q_2(0.5) = \frac{1+\sqrt{3}}{2} \approx 1.366$.
\end{itemize}
This matches the weaker R-order in~\cite{hernandez2001secant}, which assumed Hölder-continuous divided differences, but our result establishes a stronger UP-order bound based on Hölder-continuous derivatives.
\end{exam}

\begin{exam}[Muller's Method ($K=3$)]
\label{ex:muller_holder}
Muller's method (in its inverse interpolation form) uses quadratic inverse interpolation ($K=3$). The characteristic equation \eqref{eq:Kpoint_char_eq} is $q^3 = q^2 + q + \nu$. If $f \in C^{2,\nu}$, Muller's method has UP-order at least $q_3(\nu)$, where $q_3(\nu)$ is the positive root.
\begin{itemize}
 \item For $\nu=1$: $q_3(1) \approx 1.839$ (classical analysis shows this under $C^3$).
 \item For $\nu \to 0^+$: As the equation becomes $q(q^2-q-1)=0$ and $q_3(\nu)>1$ for $\nu>0$, we have $\lim_{\nu\to 0^+} q_3(\nu) = \phi \approx 1.618$.
 \item For intermediate $\nu \in (0, 1)$: $q_3(\nu)$ lies between $\phi$ and $q_3(1)$, e.g., $q_3(0.5) \approx 1.656$.
\end{itemize}
This continuous dependence of the order on $\nu$ (ranging from 1.618 to 1.839) appears to be previously unreported for $\nu < 1$.
\end{exam}

This analysis demonstrates the power of combining the UP-order framework with divided difference properties to obtain nuanced convergence rates under weaker smoothness conditions than typically assumed. \Cref{fig:multipoint} presents numerical results validating these findings. The secant method ($K=2$) was applied to $f(x) = x|x|^{\nu}/(1+\nu) + x$, which is $C^{1,\nu}$ near $x_*=0$, and the inverse Muller method ($K=3$) to $f(x) = x|x|^{1+\nu}/(2+\nu) + x$, which is $C^{2,\nu}$ near $x_*=0$, both starting from $x_0=1.0, x_1=0.8$. \Cref{fig:kpoints_normalized_error} shows the error decay $|x_k|$ for $\nu \in \{0.25, 0.5, 1.0\}$, confirming faster convergence for larger $\nu$ and for $K=3$ versus $K=2$. \Cref{fig:kpoints_convergence_rate} plots the estimated convergence order $q_k \approx \ln(|x_{k+1}|/|x_k|) / \ln(|x_k|/|x_{k-1}|)$ against the theoretical orders $q_K(\nu)$ (horizontal lines) from \eqref{eq:Kpoint_char_eq}. As observed in the figure, the estimated rates exhibit oscillations for both methods, particularly during early iterations, but converge to the sharp theoretical rates, supporting \Cref{thm:Kpoint_holder_rate}.

\begin{figure}[htb]
\centering
\begin{subfigure}[b]{0.48\textwidth}
  \centering
  \includegraphics[width=\textwidth]{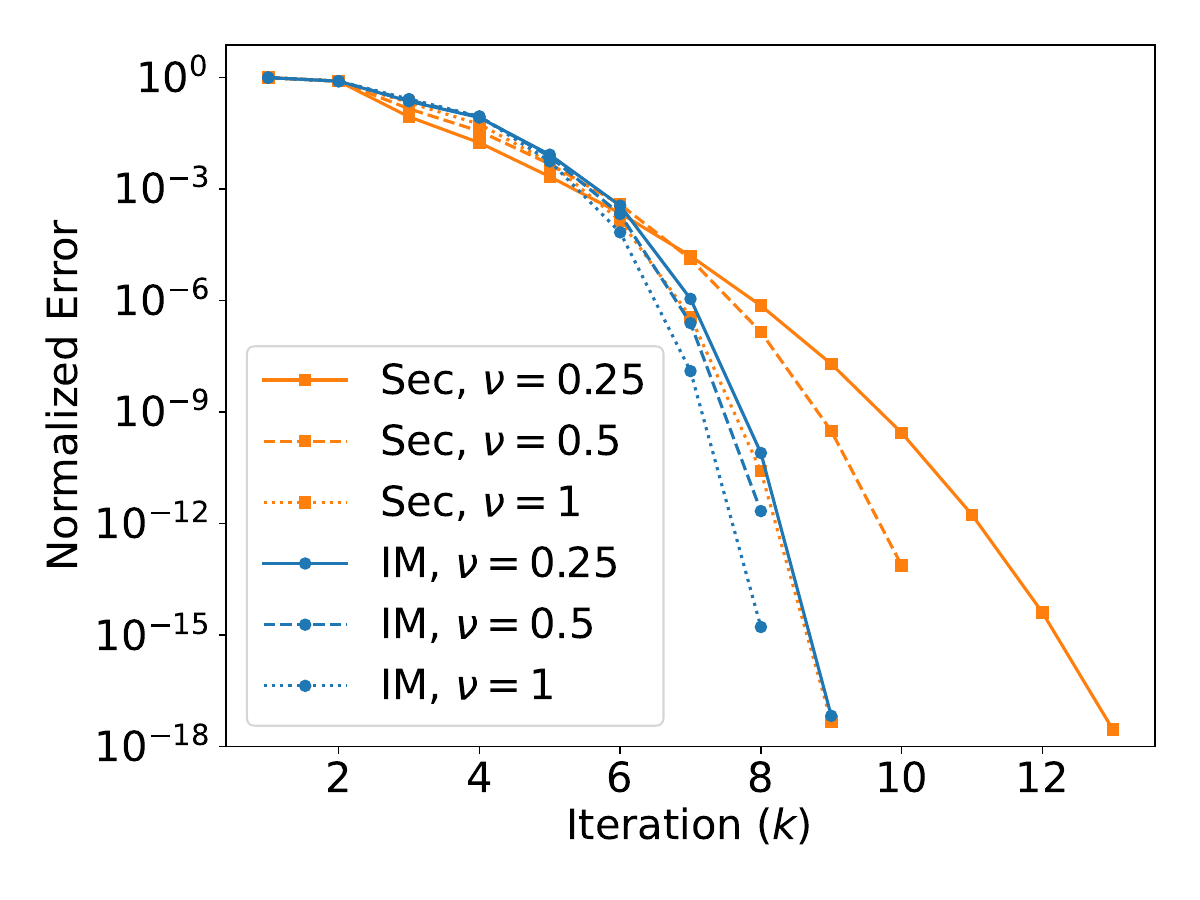}
  \caption{Error decay $|x_k|$ vs. $k$.}
  \label{fig:kpoints_normalized_error}
\end{subfigure}
\hfill 
\begin{subfigure}[b]{0.48\textwidth}
  \centering
  \includegraphics[width=\textwidth]{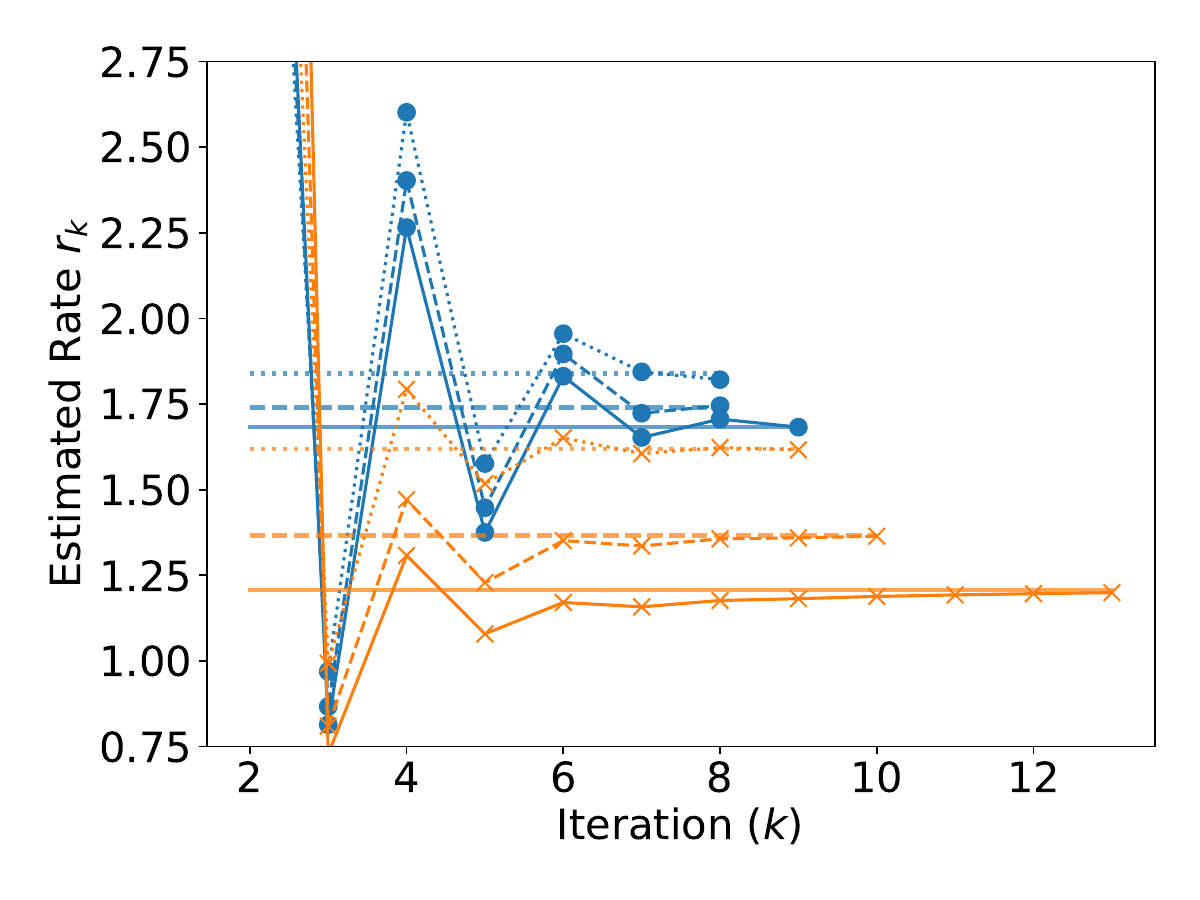}
  \caption{Estimated rate $r_k^{\text{est}}$ vs. $k$.}
  \label{fig:kpoints_convergence_rate}
\end{subfigure}
\caption{Numerical illustration of varying convergence rates of $K$-point iterative method for secant ($K=2$) and inverse Muller's ($K=3$) methods with different H\"older exponent $\nu$.}
\label{fig:multipoint}
\end{figure}

\section{Conclusions and Discussions}
\label{sec:conclusions}

This work introduced P-order, a unified convergence analysis framework designed to overcome critical limitations of classical Q-order (restrictive smoothness and norm dependence) and R-order (imprecision and ambiguity). P-order provides a necessary advancement by employing a power function $\psi(k)$ and asymptotic notation within a norm-independent structure, achieving the essential combination of rigor and granularity. This allows precise characterization of convergence rates across a wide spectrum---including logarithmic, fractional-power, linearithmic, and anti-linearithmic behaviors---while systematically handling weaker continuity conditions prevalent in modern applications by naturally aligning with appropriate Taylor remainder forms. The QUP-order and UP-order subclasses further enhance this flexibility, allowing analysis tailored to specific smoothness contexts.

The power and utility of P-order were demonstrated through novel results, not only extending convergence theory but also providing new insights. We presented a refined fixed-point analysis establishing convergence under weakened assumptions (mere differentiability via \Cref{thm:linear-superlinear}) and derived the first explicit conditions for fractional-power convergence near singularities (\Cref{thm:fp-fractional-power}). Leveraging P-order's precision, we provided the first rigorous characterization of previously unreported convergence rates for Newton's method beyond standard analytic settings (\Cref{subsec:newton_examples}) and established new fractional-power rates for gradient descent (\Cref{subsec:gd_examples}). Furthermore, the framework yielded a novel, unified analysis for $K$-point methods under relaxed $C^{K-1,\nu}$ H\"older continuity (\Cref{sec:multipoint}), culminating in the explicit rate $q_K(\nu)$ (\Cref{thm:Kpoint_holder_rate}) that precisely quantifies the impact of the H\"older exponent on convergence.

Beyond theoretical rigor, P-order offers essential tools and practical insights for algorithm design and analysis, particularly as iterative methods face complex problems in optimization and machine learning where classical assumptions fail. This refined understanding of convergence under diverse continuity conditions is crucial for guiding the development of next-generation algorithms. A key future direction is applying the insights from P-order---understanding the exact interplay between smoothness, rate functions ($\psi(k)$), and constants ($C_\psi$)---to guide the design of new or improved algorithms tailored for challenging large-scale learning scenarios.
\section*{Acknowledgments}
This work used generative AI (including Gemini 2.0 Pro and ChatGPT o3-mini) for brainstorming and improving presentation quality. All the new definitions, theorems, nontrivial proofs, and examples are the original work of the authors.

\bibliographystyle{siamplain}
\bibliography{references}

\end{document}